\documentclass{amsart}
\usepackage{amssymb,amsmath,amsthm,cite,mathtools}
\usepackage[hyphens]{url}
\usepackage[hidelinks]{hyperref}  
\hypersetup{breaklinks=true}

\usepackage{verbatim}
\usepackage{comment}
\usepackage{pstricks, pst-plot}
\numberwithin{equation}{section}
\numberwithin{figure}{section}

\DeclareMathOperator*{\supp}{supp}
\newcommand{\abs}[1]{\left\vert#1\right\vert}
\newcommand{\Vrt}[1]{\left\Vert #1 \right\Vert}

\newcommand{\cl}[1]{\textup{cl}({#1})}

\newcommand{\COMMENT}[1]{}
\newcommand*{\quark}{\setbox0\hbox{$x$}\hbox to\wd0{\hss$\cdot$\hss}}
\newcommand*{\rd}{\mathrm{d}}
\usepackage{mathtools}
\newcommand*{\defeq}{\coloneqq}

\newcommand{\todo}[1]{\bgroup\color{red}[#1]\egroup}

\newcommand{\allok}{\bgroup\color{purple}\ensuremath{[\checkmark]}\egroup}

\newcommand{\hcl}[1]{\bgroup\color{green!50!black}[#1]\egroup}
\newcommand{\tjs}[1]{\bgroup\color{blue}[#1]\egroup}

\theoremstyle{plain}
\newtheorem{theorem}{Theorem}[section]
\newtheorem{lemma}[theorem]{Lemma}
\newtheorem{proposition}[theorem]{Proposition}
\newtheorem{corollary}[theorem]{Corollary}
\newtheorem{assumption}[theorem]{Assumption}

\theoremstyle{definition}
\newtheorem{definition}[theorem]{Definition}
\newtheorem{example}[theorem]{Example}

\newtheorem{remark}[theorem]{Remark}

\begin{document}

\title[Strong and weak modes of measures on vector spaces]{Equivalence of weak and strong modes \\of measures on topological vector spaces}

\begin{abstract}
	A strong mode of a probability measure on a normed space $X$ can be defined as a point $u$ such that the mass of the ball centred at $u$ uniformly dominates the mass of all other balls in the small-radius limit.
	Helin and Burger weakened this definition by considering only pairwise comparisons with balls whose centres differ by vectors in a dense, proper linear subspace $E$ of $X$, and posed the question of when these two types of modes coincide.
	We show that, in a more general setting of metrisable vector spaces equipped with measures that are finite on bounded sets, the density of $E$ and a uniformity condition suffice for the equivalence of these two types of modes.
	We accomplish this by introducing a new, intermediate type of mode.
	We also show that these modes can be inequivalent if the uniformity condition fails.
	Our results shed light on the relationships between among various notions of maximum a posteriori estimator in non-parametric Bayesian inference.
\end{abstract}

\author{Han Cheng Lie}
\address{Han Cheng Lie\\
	Institut f\"{u}r Mathematik\\
	Freie Universit\"{a}t Berlin\\
	14195 Berlin\\
	Germany}
\curraddr{Zuse Institute Berlin\\
	Takustr.\ 7\\
	14195 Berlin\\
	Germany}
\email{hlie@math.fu-berlin.de}

\author{T.~J.~Sullivan}
\address{T.~J.~Sullivan\\
	Freie Universit{\"a}t Berlin and Zuse Institute Berlin\\
	Takustr.\ 7\\
	14195 Berlin\\
	Germany}
\email{sullivan@zib.de}

\date{\today}

\subjclass[2010]{
	28C20 
	(62G35 
	49Q20 
	49N45) 
}

\keywords{mode, small ball probabilities, topological vector space, maximum a posteriori estimation, minimum action principle}

\maketitle

\section{Introduction}
\label{sec:introduction}

A probability measure $\mu$ on a topological vector space $X$ can be described using summary statistics such as means and covariances or, as in this article, \emph{modes}, meaning points of maximum $\mu$-mass in an appropriate sense.
There are multiple ways of defining such modes, particularly for infinite-dimensional vector spaces; the \emph{strong mode} of Dashti et al.~\cite{dashti_law_stuart_voss_map_estimators_2013_published} and the \emph{weak mode} of Helin and Burger~\cite{helin_burger_map_estimates_2015_published} describe the notion of a point of maximum probability in a separable Banach space by examining the probabilities of norm balls in the small-radius limit.

In the setting of Bayesian inverse problems, the modes of the posterior are known as \emph{maximum a posteriori estimators} or \emph{MAP points}).
MAP points play an important role because of their interpretation as being the most likely solutions of the inverse problem given the observed data \cite{stuart_inverse_problems}.
A far-from-exhaustive list of studies that have considered MAP points is \cite{bui-thanh-nguyen2016,yao_hu_li_2016,dunlop_stuart_2016,alexanderian_petra_stadler_ghattas,hyvonen_leinonen,burger_lucka,calatroni_delosreyes_schoenlieb,pereyra,heas_lavancier_kadri_harouna,harhanen_hyvonen_majander_staboulis,liu_zheng,calvetti_hakula_pursiainen_somersalo,hamalainen_kallonen_kolehmainen_lassas_niinimaki_siltanen,vaksman_zibulevsky_elad,krol_li_shen_xu,helin_lassas}.
Modes play an important role in other areas of applied mathematics as well.
For example, in mathematical models of chemical reactions, rare events of diffusion processes on energy landscapes play an important role;
the rare events of interest are typically transitions from one energy well or metastable state to another.
The qualitative description of such rare events requires studying the paths that a diffusion process is most likely to take.
Studying the modes of the law of the diffusion process on the associated path space is often done with the help of a least-action principle formulated by an Onsager--Machlup functional, and an application of Freidlin--Wentzell theory or large deviations theory \cite{freidlinwentzell, dembozeitouni, e_ren_vandeneijnden,lu_stuart_weber}.

Onsager--Machlup functionals --- and, more generally, variational approaches --- play an important role in Bayesian inverse problems.
When $\mu$ is a Bayesian re-weighting of a Gaussian prior, the strong mode of Dashti et al.\ (see \cite[Definition~3.1]{dashti_law_stuart_voss_map_estimators_2013_published} or our Definition~\ref{def:mode}) is the minimiser of an appropriate Onsager--Machlup functional \cite[Theorem~3.5]{dashti_law_stuart_voss_map_estimators_2013_published}.
This functional can be formally regarded as a misfit regularised by the negative logarithm of the prior density.
Under the assumptions considered in \cite{helin_burger_map_estimates_2015_published}, the weak mode of Helin and Burger (see \cite[Definition~4]{helin_burger_map_estimates_2015_published} or our Definition~\ref{def:Eweak_mode}) corresponds to the points where certain logarithmic derivatives of $\mu$ vanish; see \cite[Theorems~2 and 3]{helin_burger_map_estimates_2015_published}.
Recently, Agapiou et al.\ \cite{agapiou_et_al_sparsity} resolved one of the questions raised by Helin and Burger, by extending the weak mode formalism and variational approach to analyse $L^1$-priors on Besov spaces; these priors are perceived as having sparsity-promoting properties advantageous to inverse problems and imaging.

Helin and Burger \cite{helin_burger_map_estimates_2015_published} asked under what circumstances a weak mode is also strong mode.
While it followed immediately from the definitions that a strong mode is a weak mode, the converse was far from clear.
Their motivation for the question would be to determine whether the weak mode was a genuinely different type of mode.
Another motivation for demonstrating the equivalence of strong and weak modes is that, if it were known that every weak mode is a strong mode, then one could apply the variational characterisation of weak modes in terms of logarithmic derivatives in order to obtain the strong modes of Dashti et al.
In particular, one could enlarge the available set of tools for identifying strong modes --- and thus MAP estimators for Bayesian inverse problems --- by drawing upon the powerful theory of differentiable measures \cite{bogachev_dmmc}.

In this work, we formulate the question of when a weak mode is a strong mode in a more general context than that considered in \cite{helin_burger_map_estimates_2015_published}.
We present two conditions that jointly suffice for the equivalence of strong and weak modes:
a uniformity condition \eqref{eq:uniformity_condition}, and the topological density of the set $E$ with which one defines the weak mode.
We introduce an intermediate type of mode and reduce the original question to the task of identifying sufficient conditions for two simpler equivalence statements, presented as Theorem~\ref{thm:equivalence_Estrong_strong} and Theorem~\ref{thm:equivalence_Eweak_Estrong} below.
We state the necessary definitions and results in Section~\ref{sec:main_definitions_results}, and prove Theorems~\ref{thm:equivalence_Estrong_strong} and \ref{thm:equivalence_Eweak_Estrong} in Section~\ref{sec:E_strong_equiv_strong} and Section~\ref{sec:E_weak_equiv_E_strong} respectively.
We also illustrate the importance of the uniformity condition \eqref{eq:uniformity_condition} by giving in Example~\ref{example:weak_does_not_imply_weak_uniform} an example of a measure that does not satisfy the uniformity condition and has weak modes but no strong modes.

The significance of Theorems~\ref{thm:equivalence_Estrong_strong} and \ref{thm:equivalence_Eweak_Estrong} for inverse problems is that, at the level of generality in which Helin and Burger formulated their question, weak modes and strong modes are not equivalent.
Some of the assumptions therein, e.g.\ the existence of representatives of Radon--Nikodym derivatives that are continuous on $X$, are neither sufficient nor necessary.
In addition, our main results show that one can consider weak modes for a much larger class of measures, since we make no assumptions about differentiability of $\mu$.
When one wishes to compute modes, then such assumptions may be useful for the variational characterisation of weak modes as zeros of logarithmic derivatives, but they are not relevant for the purpose of determining the equivalence of strong and weak modes.

In the context of inverse problems, our results provide theoretical support for certain choices of negative log-likelihood or potential functions, in which the Onsager--Machlup functional assumes the value $+\infty$ outside a topologically dense subset of $X$ \cite[Equation (2.1)]{dashti_law_stuart_voss_map_estimators_2013_published}.
Our results apply not only to normed spaces, but also to metric spaces such as $F$-spaces, so that one may apply the MAP approach in function spaces that have desirable sparsity properties, e.g.\ $L^{p}$ for $0 < p < 1$.
We believe that the study of MAP points on more general spaces than normed vector spaces will prove to be relevant in future research on Bayesian inverse problems, and that our results provide a useful first step in this direction.

In Section~\ref{sec:discussion}, we consider some extensions of our results.
We show that modes can be defined for complex vector spaces in Section~\ref{ssec:arbitrary_nonzero_scaling_factors}, and demonstrate that similar equivalence statements hold for the \textit{local} modes defined by Agapiou et al.\ \cite{agapiou_et_al_sparsity}.
In Section~\ref{ssec:examples_concerning_existence_K-dependence_modes}, we show that a probability measure on a topological vector space need not have a mode (Example~\ref{eg:no_mode}), and we show that the set of modes can depend on the choice of the set $K$ (Example~\ref{eg:mode_depends_on_K}).
The results of the latter section indicate that the study of modes contains subtle surprises.

\section{Main results}
\label{sec:main_definitions_results}

In this paper, $\mathbb{N} \defeq \{ 1, 2, \dots, \}$, $\mathbb{N}_{0} \defeq \mathbb{N} \cup \{ 0 \}$, and $\mathbb{K} \defeq \mathbb{R}$ or $\mathbb{C}$.
We shall work in the following setting:

\begin{assumption}
	\label{asmp:context}
	$X$ is a first countable, Hausdorff topological vector space over $\mathbb{K}$; $\mu$ is a non-zero measure on the Borel $\sigma$-algebra of $X$ that is finite on bounded sets; and $K$ is a bounded, open neighbourhood of the origin.
\end{assumption}

In \cite{dashti_law_stuart_voss_map_estimators_2013_published,helin_burger_map_estimates_2015_published}, $X$ is a separable Banach space, $K$ is its unit norm ball, and $\mu$ is a Borel probability measure with topological support $\supp(\mu) = X$.
Here, we do not assume that $X$ is complete or separable, that $\supp(\mu)=X$, or that $\mu$ is finite on $X$.
The existence of a set $K$ satisfying Assumption~\ref{asmp:context} implies that $X$ is metrisable \cite[Theorem~6.2.1]{narici_beckenstein}.

For any $\lambda \in \mathbb{K}$, $x \in X$, and $E\subset X$, $x + \lambda E \defeq \{x + \lambda y \mid y \in E\}$.
We denote the collection of open neighbourhoods of $x \in X$ by $\mathcal{N}(x)$.
For any $U\in\mathcal{N}(0)$, let $U(x,\lambda)\defeq x+\lambda U$.
For $K$ as in Assumption~\ref{asmp:context}, define the evaluation map $f \colon X\times\mathbb{R}_{>0}\to\mathbb{R}_{\geq 0}$ by
\begin{equation}
	\label{eq:evaluation_map_f}
	f(x, r)\defeq \mu(K(x, r)) \geq 0.
\end{equation}
In Section~\ref{ssec:arbitrary_nonzero_scaling_factors}, we consider the more general case of nonzero (but not necessarily strictly positive) radius parameter $r$. This is important in situations when a problem is formulated over a complex topological vector space $X$: in such cases, admitting only real, strictly positive values of the radius parameter $r$ would limit the ensuing analysis to the real restriction of $X$.
Assumption~\ref{asmp:context} implies that $f(x, r)$ is finite for all $(x, r) \in X \times \mathbb{R}_{>0}$.
Below, we reformulate \cite[Definition 3.1]{dashti_law_stuart_voss_map_estimators_2013_published} of a MAP estimator for $\mu$.

\begin{definition}\label{def:mode}
	A \emph{strong mode} (or simply \emph{mode}) of $\mu$ is any $u \in X$ satisfying
	\begin{equation}
		\label{eq:mode}
		\lim_{r \downarrow 0} \frac{\sup_{z \in X} f(z, r)}{f(u, r)}= 1.
	\end{equation}
\end{definition}

The intuition behind \eqref{eq:mode} is that, if $u$ is a mode, then, for sufficiently small $r$, translating the $r K$-neighbourhood of $u$ does not yield an increase in measure.

Next, we reformulate \cite[Definition 4]{helin_burger_map_estimates_2015_published} of a weak MAP estimator for $\mu$.

\begin{definition}
	\label{def:Eweak_mode}
	Let $E\subset X$.
	An $E$-\emph{weak mode} (or simply \emph{weak mode}) of $\mu$ is any $u \in\supp(\mu)$ such that
	\begin{equation}
		\label{eq:Eweak_mode}
		\lim_{r \downarrow 0} \frac{f(u - v, r)}{f(u, r)} \leq 1 ,\quad\forall v\in E.
	\end{equation} 
\end{definition}

A weak mode $u$ differs from a strong mode in that it considers translations of $K(u, r)$ only by vectors in $E$;
thus every mode is an $E$-weak mode for every $E\subset X$.
We require $u\in\supp(\mu)$ in order to exclude ``$0/0$ problems'':
for example, if $\mu \defeq \tfrac{1}{2} (\mu_1 + \mu_2)$, where $\mu_1$ (resp.\ $\mu_2$) is the standard one-dimensional Gaussian measure on $\{ + 1 \} \times \mathbb{R}$ (resp.\ $\{ - 1 \} \times \mathbb{R}$), and $E \defeq \mathop{\mathrm{span}}(e_{2})$, then $0 \notin \supp(\mu)$ is an $E$-weak mode in the $0/0$ sense.

We now introduce a new definition.
\begin{definition}
	\label{def:supE_weak_mode}
	Let $E\subset X$.
	An $E$-\emph{strong mode} of $\mu$ is any $u \in\supp(\mu)$ such that
	\begin{equation}
		\label{eq:supE_weak_mode}
		\lim_{r \downarrow 0} \frac{\sup_{z \in u - E} f(z, r)}{f(u, r)} \leq 1.
	\end{equation}
\end{definition}

Note that an $X$-strong mode is simply a strong mode.
By comparing \eqref{eq:mode}, \eqref{eq:Eweak_mode}, and \eqref{eq:supE_weak_mode}, and observing that $u-E\subset X$ for any $u\in X$ and $E\subset X$, we have the following chain of implications:
\begin{equation}
	\label{eq:strong_implies_supEweak_implies_Eweak}
	\text{Strong}~\implies~\text{$E$-strong}~\implies~\text{$E$-weak}
\end{equation}

Helin and Burger asked when an $E$-weak mode (for a topologically dense subspace $E$ of some Banach space $X$) is also a strong mode.
In this article, we address their question by investigating the converses of the implications in \eqref{eq:strong_implies_supEweak_implies_Eweak}.

\begin{theorem}
	\label{thm:equivalence_Estrong_strong}
	Under Assumption~\ref{asmp:context}, if $E$ is topologically dense in $X$, then $u$ is an $E$-strong mode if and only if $u$ is a strong mode.
\end{theorem}

\begin{theorem}
	\label{thm:equivalence_Eweak_Estrong}
	Suppose that Assumption~\ref{asmp:context} holds, and let $E$ be a nonempty subset.
	If the pair $(u,E)$ satisfies the \emph{uniformity condition}
	\begin{align}
		\label{eq:uniformity_condition}
		\exists (v^{\ast},r^{\ast})\in E\times (0,1) &\text{ such that, for all $r \in (0, r^{\ast})$,} \\
		\notag
		f(u-v^\ast,r) &=\sup_{z\in u-E}f(z,r),
	\end{align}
	then $u$ is an $E$-weak mode if and only if $u$ is an $E$-strong mode.
\end{theorem}
Note that neither Theorem~\ref{thm:equivalence_Estrong_strong} nor Theorem~\ref{thm:equivalence_Eweak_Estrong} make any claim about the existence of modes of any type.
In addition, neither of these theorems requires $E$ to be a subspace of $X$, unlike the assumption on \cite[p.~2]{dashti_law_stuart_voss_map_estimators_2013_published} or \cite[Assumption (A1)]{helin_burger_map_estimates_2015_published}.

Theorems~\ref{thm:equivalence_Estrong_strong} and \ref{thm:equivalence_Eweak_Estrong} yield a positive answer to Helin and Burger's question:

\begin{corollary}
	\label{cor:main_result}
	Under Assumption~\ref{asmp:context}, if $E$ is topologically dense in $X$, and if $(u,E)$ satisfies the uniformity condition \eqref{eq:uniformity_condition}, then $u$ is an $E$-weak mode if and only if $u$ is a strong mode.
\end{corollary}

On the other hand, \emph{without} the uniformity condition \eqref{eq:uniformity_condition},  Example~\ref{example:weak_does_not_imply_weak_uniform} shows that the answer to Helin and Burger's question is generally \emph{negative}, i.e.\ there exist finite Borel measures --- even on $X = \mathbb{R}^{2}$ --- with weak modes but no strong modes.

\section{The equivalence of $E$-strong and strong modes}
\label{sec:E_strong_equiv_strong}

In this section, we seek conditions for which \eqref{eq:supE_weak_mode} implies \eqref{eq:mode}:
\begin{equation*}
	\lim_{r\downarrow 0}\frac{\sup_{z\in u-E}f(z,r)}{f(u,r)}\leq 1 \implies \lim_{r\downarrow 0}\frac{\sup_{z\in X}f(z,r)}{f(u,r)}=1.
\end{equation*}
In Lemma~\ref{lem:lower_semicontinuity_of_f}, we show that the evaluation map $f$ defined in \eqref{eq:evaluation_map_f} is lower semicontinuous.
We use this powerful property to derive the second key result of this section, Proposition~\ref{prop:E_topologically_dense_implies_suprema_of_f_on_r_fibres_coincide}, which shows that topological density of $E$ suffices for every $E$-strong mode to be a strong mode.
By \eqref{eq:strong_implies_supEweak_implies_Eweak}, this yields the conclusion of Theorem~\ref{thm:equivalence_Estrong_strong}.

We begin this section with a useful definition.
\begin{definition}
	\label{def:coincident_limiting_ratios_condition}
	The \emph{coincident limiting ratios} condition holds for $x $ and $E \subset X$ if
	\begin{equation}
		\label{eq:coincident_limiting_ratios_condition}
		\tag{CLR}
		\lim_{r \downarrow 0} \frac{\sup_{z \in x - E} f(z, r)}{f(x, r)} = \lim_{r \downarrow 0} \frac{\sup_{z \in X} f(z, r)}{f(x, r)}.
	\end{equation}
\end{definition}

\begin{lemma}
	\label{lem:equivalent_definitions_of_weak_modes}
	If $u$ is an $E$-strong mode, and if $u$ and $E$ satisfy \eqref{eq:coincident_limiting_ratios_condition}, then $u$ is a strong mode.
\end{lemma}

\begin{proof}
	By Definition~\ref{def:supE_weak_mode} and \eqref{eq:coincident_limiting_ratios_condition}, the limit in \eqref{eq:mode} is less than or equal to 1.
	Since the limit cannot be strictly less than 1, the result follows.
\end{proof}

We now show the lower semicontinuity of the evaluation map $f(\quark,r) \colon X\to\mathbb{R}_{\geq 0}$, for any $r>0$.
We shall use this property to prove Theorem~\ref{thm:equivalence_Estrong_strong}.
\begin{lemma}
	\label{lem:lower_semicontinuity_of_f}
	For arbitrary $r>0$, $f(\quark, r)$ is lower semicontinuous on $X$.
\end{lemma}

\begin{proof}
	Fix an arbitrary $x \in X$ and an arbitrary sequence $(x_{n})_{n \in \mathbb{N}}$ converging to $x$.
	To prove the desired statement, it suffices to show that 
	\begin{equation}
		\label{eq:indicator_Kxr_of_y_leq_liminf_indicator_Kxnr_of_y_for_all_y}
		\forall y \in X, \quad
		\mathbb{I}_{K(x, r)}(y) \leq \liminf_{n\to\infty} \mathbb{I}_{K(x_{n}, r)}(y),
	\end{equation}
	by considering separately the cases of when $y \in K(x, r)$, when $y \notin \cl{K(x, r)}$, and when $y \in \partial K(x, r)$.
	Given \eqref{eq:indicator_Kxr_of_y_leq_liminf_indicator_Kxnr_of_y_for_all_y}, it follows from \eqref{eq:evaluation_map_f} and Fatou's lemma that
	\begin{equation*}
		f(x, r) \leq \int_{X}\liminf_{n\to\infty}\mathbb{I}_{K(x_{n}, r)}(y) \, \mu(\rd y) \leq \liminf_{n\to\infty} \int_{X}\mathbb{I}_{K(x_{n}, r)}(y) \, \mu(\rd y).
	\end{equation*}
	Since the rightmost term is equal to $\liminf_{n\to\infty} f(x_{n}, r)$, it follows that $f(\quark, r)$ is lower semicontinuous at $x$, for arbitrary $x \in X$.
	It remains only to establish \eqref{eq:indicator_Kxr_of_y_leq_liminf_indicator_Kxnr_of_y_for_all_y}.

	Let $y \in K(x, r)$ be arbitrary.
	Since $y \in K(x, r)$, and since $K(x, r)$ is open, there exists a bounded $U\in\mathcal{N}(0)$ such that, for all $v \in U$, $y - v \in K(x, r)$.
	Since $y - v \in K(x, r)$ if and only if $y \in K(x + v, r)$, it follows that $z \in (x + U) \implies y \in K(z, r)$.
	Since $x_{n} \to x$ as $n \to \infty$, there exists some $N \in \mathbb{N}$ such that $x_{n} \in (x + U)$ for all $n \geq N$.
	Then $\mathbb{I}_{K(x_{n}, r)}(y)=1$ for $n \geq N$, and
	\begin{equation}
		\label{eq:prel01}
		\forall y \in K(x, r), \quad
		\mathbb{I}_{K(x_{n}, r)}(y) \to \mathbb{I}_{K(x, r)}(y).
	\end{equation}
	Now let $y \notin \cl{K(x, r)}$ be arbitrary.
	Since $\cl{K(x, r)}$ is closed, there exists a bounded $U'\in\mathcal{N}(0)$ such that, for all $v \in U'$, $y - v \notin \cl{K(x, r)}$.
	Since $y - v \notin \cl{K(x, r)}$ if and only if $y \notin \cl{K(x + v, r)}$, it follows that $z \in (x + U') \implies y \notin \cl{K(z, r)}$.
	Since $x_{n} \to x$, there exists some $N \in \mathbb{N}$ such that $x_{n} \in x + U'$ for all $n \geq N$.
	This implies that $\mathbb{I}_{K(x_{n}, r)}(y)=0$ for all $n \geq N$.
	Thus,
	\begin{equation}
		\label{eq:prel02}
		\forall y \notin\cl{K(x, r)}, \quad
		\mathbb{I}_{K(x_{n}, r)}(y) \to \mathbb{I}_{K(x, r)}(y).
	\end{equation}
	Observe that $\liminf_{n\to\infty}\mathbb{I}_{K(x_{n}, r)}(y)$ is either 0 or 1, because $\mathbb{I}_{K(x_{n}, r)}(y)$ is either 0 or 1.
	On the other hand, since $K(x, r)$ is open, $\mathbb{I}_{K(x, r)}(y) = 0$ for every $y \in \partial K(x, r)$.
	Thus, $\mathbb{I}_{K(x, r)} \leq \liminf_{n\to\infty} \mathbb{I}_{K(x_{n}, r)}$ on $\partial K(x, r)$.
\end{proof}

We now use lower semicontinuity to show that topological density of $E$ in $X$ suffices for any pair $(x,E)$ to satisfy \eqref{eq:coincident_limiting_ratios_condition}.

\begin{proposition}
	\label{prop:E_topologically_dense_implies_suprema_of_f_on_r_fibres_coincide}
	Suppose that Assumption~\ref{asmp:context} holds.
	If $E$ is topologically dense in $X$, then $\sup_{z \in X} f(z, r) = \sup_{z \in x - E} f(z, r)$ for all $(x, r)$.
	In particular, \eqref{eq:coincident_limiting_ratios_condition} holds for any $x\in X$.
\end{proposition}

\begin{proof}
	The second conclusion follows immediately from the first, so it suffices to prove the first conclusion.
	Let $(x, r)\in X\times \mathbb{R}_{>0} $ be arbitrary, and suppose that $\sup_{z\in X}f(z,r)=+\infty$.
	Let $M\in\mathbb{N}$ be arbitrary.
	Recall that Lemma~\ref{lem:lower_semicontinuity_of_f} yields the lower semicontinuity of $f(\quark,r)$.
	Given lower semicontinuity of $f(\quark,r)$, the set $\{z\in X\ \vert\ f(z,r)>M\}$ is open; the set is also nonempty since $\sup_{z\in X}f(z,r)=+\infty$.
	Since the topological density of $E$ in $X$ implies the topological density of $x-E$ in $X$, it follows that there exist $z\in x-E$ with $f(z,r)>M$.
	Since $M$ was arbitrary, it follows that $\sup_{z\in x-E}f(z,r)=+\infty$ as well.
	
	Now suppose that $\sup_{z\in X}f(z,r)$ is finite, and let $(x_{n})_{n\in\mathbb{N}}\subset X$ be such that $f(x_{n},r)\in\mathbb{R}$ and $f(x_{n},r)\to \sup_{z\in X}f(z,r)$.
	It holds that $\{z\in X\ \vert\ f(z,r)>f(x_{n},r)\}$ is open and nonempty.
	Since $x-E$ is topologically dense in $X$, there exist $z\in x-E$ such that $f(z,r)>f(x_{n},r)$.
	Since $f(x_{n},r)\to \sup_{z\in X}f(z,r)$, this implies that $\sup_{z\in x-E}f(z,r)\geq \sup_{z\in X}f(z,r)$.
\end{proof}

Theorem~\ref{thm:equivalence_Estrong_strong} follows from the preceding observations.

\begin{proof}[Proof of Theorem~\ref{thm:equivalence_Estrong_strong}]
	By \eqref{eq:strong_implies_supEweak_implies_Eweak}, every strong mode is an $E$-strong mode for any subset $E \subset X$.
	For the converse, since $E$ is topologically dense in $X$, Proposition~\ref{prop:E_topologically_dense_implies_suprema_of_f_on_r_fibres_coincide} implies that \eqref{eq:coincident_limiting_ratios_condition} holds for $u$ and $E$.
	Applying the hypothesis that $u$ is an $E$-strong mode and Lemma~\ref{lem:equivalent_definitions_of_weak_modes} completes the proof.
\end{proof}

For the sake of completeness, we show that, under the hypotheses that $E$ is topologically dense in $X$ and $u$ is an $E$-strong mode, it holds that $u$ is a strong mode if and only if the coincident limiting ratios condition holds for $u$ and $E$.
We shall use the following lemma:

\begin{lemma}\label{lem:u_minus_E_limiting_ratio_geq_1_and_mode_implies_coincident_limiting_ratios}
	Let $u$ be an $E$-strong mode, and suppose that $u$ and $E$ satisfy
 	\begin{equation}
		\label{eq:u_minus_E_limiting_ratio_geq_1}
		1\leq \lim_{r \downarrow 0} \frac{\sup_{z \in u - E} f(z, r)}{f(u, r)}.
	\end{equation}
	If $u$ is also a mode, then $u$ satisfies \eqref{eq:coincident_limiting_ratios_condition}.
\end{lemma}

\begin{proof}
	We prove the contrapositive.
	Suppose that \eqref{eq:coincident_limiting_ratios_condition} does not hold; then
	\begin{equation*}
		1\leq \lim_{r \downarrow 0} \frac{\sup_{z \in x - E} f(z, r)}{f(x, r)} < \lim_{r \downarrow 0} \frac{\sup_{z \in X} f(z, r)}{f(x, r)}
	\end{equation*}
	and by \eqref{eq:mode} it follows that $u$ is not a mode.
\end{proof}

To establish a sufficient condition for the hypothesis \eqref{eq:u_minus_E_limiting_ratio_geq_1} to hold, we use the lower semicontinuity of $f(\quark,r)$ for any $r>0$, and the fact that Assumption~\ref{asmp:context} implies $X$ is a metrisable space.

\begin{proposition}
	\label{prop:W_E_geq_1}
	Let $u \in \supp(\mu)$ and $\emptyset \neq E \subset X$.
	Suppose Assumption~\ref{asmp:context} holds.
	If either $E$ contains the origin or is topologically dense in a neighbourhood of the origin, then $u$ and $E$ satisfy \eqref{eq:u_minus_E_limiting_ratio_geq_1}.
\end{proposition}

\begin{proof}
	Suppose $0\in E$.
	Then $u \in u - E$, and $f(u, r) \leq \sup_{z \in u - E} f(z, r)$ for all $r$.
	
	Suppose that $E$ is topologically dense in a neighbourhood $V$ of the origin.
	Then $u - E$ is dense in $u - V$, and there exists a sequence of points in $(u - E)\cap (u - V)$ that converges to $u$.
	Let $r>0$ be arbitrary.
	By Lemma~\ref{lem:lower_semicontinuity_of_f}, it follows that $f(\quark,r)$ is lower semicontinuous at $u$.
	Using lower semicontinuity of $f(\quark, r)$ at $u$, and using the fact that $X$ is a metrisable space, it follows that $\liminf_{n\to\infty} f(u_{n},r)\geq f(u,r)$.
	Since $(u_{n})_{n\in\mathbb{N}}\subset u-E$, it follows that $\sup_{z\in u-E}f(z,r)\geq f(u,r)$.
\end{proof}

\begin{corollary}
	Let $E$ be topologically dense in $X$, and let $u\in\supp(\mu)$ be an $E$-strong mode.
	Then $u$ is a strong mode if and only if \eqref{eq:coincident_limiting_ratios_condition} holds for $u$ and $E$.
\end{corollary}

\begin{proof}
	Sufficiency of \eqref{eq:coincident_limiting_ratios_condition} for $u$ to be a mode follows from Lemma~\ref{lem:equivalent_definitions_of_weak_modes}.
	Necessity follows from Lemma~\ref{lem:u_minus_E_limiting_ratio_geq_1_and_mode_implies_coincident_limiting_ratios}, which we may apply given that the topological density of $E$ ensures that we may apply Proposition~\ref{prop:W_E_geq_1}.
\end{proof}

\section{The equivalence of $E$-weak and $E$-strong modes}
\label{sec:E_weak_equiv_E_strong}

In this section, we seek conditions for which \eqref{eq:Eweak_mode} implies \eqref{eq:supE_weak_mode}, i.e.
\begin{equation*}
	\left( \lim_{r\downarrow 0}\frac{f(u-v,r)}{f(u,r)}\leq 1,\quad\forall v\in E \right) \implies \lim_{r\downarrow 0}\frac{\sup_{z\in u-E}f(z,r)}{f(u,r)}\leq 1.
\end{equation*}
A recurring observation in this section is that finding general conditions for the relation above to hold is relatively easier when $X$ is finite-dimensional, compared to when $X$ is infinite-dimensional.
When $X$ is infinite-dimensional, one such condition is given by the uniformity condition \eqref{eq:uniformity_condition}, which has the interpretation that there exists some point $u-v^\ast \in u-E$ such that, for all sufficiently small $r$, the $\mu$-measure of $K(u-v^\ast,r)$ dominates that of $K(u-v,r)$ for all $v\in E$. 

The following proposition provides sufficient conditions for a weak mode to be $E$-strong, in the case when $X$ is a finite-dimensional vector space:

\begin{proposition}
	\label{prop:equivalent_definitions_of_weak_modes_finite_dim}
	Let $X=\mathbb{R}^n$, let $E$ be a dense proper subset of $X$, let $K(0,1)=B_p(0,1)$ be the unit $\ell_p$-ball centred at the origin for some $0<p\leq \infty$, and let $\mu$ admit a continuous density $g$ with respect to the Lebesgue measure $\lambda$.
	If $u$ is an $E$-weak mode and $g(u)>0$, then $u$ is a strong mode, and thus also an $E$-strong mode.
\end{proposition}

\begin{remark}
	\label{rem:condition_on_set_K_for_finite_dimensions}
	Note that the restriction to $\ell_p$-balls is unnecessary. We may choose $K(0,1)$ to be any bounded set, since the family of $r$-dilates of $K(0,1)$ will satisfy the property of bounded eccentricity; this property suffices for the application of the Lebesgue differentiation theorem \cite[Chapter 3, Corollary 1.7]{stein}.
\end{remark}

\begin{proof}[Proof of Proposition~\ref{prop:equivalent_definitions_of_weak_modes_finite_dim}]
	Fix an arbitrary $v\in E$.
	Observe that
	\begin{equation*}
		\lim_{r\downarrow 0}\frac{f(u-v,r)}{f(u,r)}=\lim_{r\downarrow 0}\frac{\mu(B_2(u-v,r))}{\mu(B_2(u,r))}=\lim_{r\downarrow 0}\frac{\mu(B_2(u-v,r))}{\lambda(B_2(u-v,r))}~\frac{\lambda(B_2(u-v,r))}{\mu(B_2(u,r))}.
	\end{equation*}
	Since the Lebesgue measure is translation-invariant, it follows that $\lambda(B_2(u-v,r))=\lambda(B_2(u,r))$.
	Using this and the product rule for convergent real sequences, we have 
	\begin{equation*}
		\lim_{r\downarrow 0}\frac{f(u-v,r)}{f(u,r)}=\lim_{r\downarrow 0}\frac{\mu(B_2(u-v,r))}{\lambda(B_2(u-v,r))}~\lim_{r\downarrow 0}\frac{\lambda(B_2(u,r))}{\mu(B_2(u,r))}.
	\end{equation*}
	Thus, by the quotient rule for convergent sequences, the hypothesis that $g(u)>0$, and the Lebesgue differentiation theorem, we obtain
	\begin{equation*}
		\lim_{r\downarrow 0}\frac{\lambda(B_2(u,r))}{\mu(B_2(u,r))}=\left(\lim_{r\downarrow 0}\frac{\mu(B_2(u,r))}{\lambda(B_2(u,r))}\right)^{-1}=\left(g(u)\right)^{-1}.
	\end{equation*}
	Applying the Lebesgue differentiation theorem again and using the hypothesis that $u$ is an $E$-weak mode, we thus obtain
	\begin{equation*}
		1\geq \lim_{r\downarrow 0}\frac{f(u-v,r)}{f(u,r)}=\frac{g(u-v)}{g(u)},\quad\forall v\in E.
	\end{equation*}
	The above implies that 
	\begin{equation}
		\label{eq:consequence_of_translation_invariance_1}
		\sup_{z\in u-E}g(z)\leq g(u).
	\end{equation}
	By continuity of $g$ and density of $E$ in $X$, \eqref{eq:consequence_of_translation_invariance_1} implies the stronger statement that $\sup_{z\in X}g(z)\leq g(u)$.
	Translation invariance of Lebesgue measure yields
	\begin{equation*}
	\sup_{z\in X}f(z,r)\leq \sup_{z\in X} g(z)~\lambda(B_2(u',r)),
	\end{equation*}
	for any $u'\in X$.
	Using the properties of limits, the definition of $f$, the Lebesgue differentiation theorem, and \eqref{eq:consequence_of_translation_invariance_1}, we have
	\begin{align*}
		\lim_{r\downarrow 0}\frac{\sup_{z\in X}f(z,r)}{f(u,r)}&\leq \sup_{z\in X}g(z)\left(\lim_{r\downarrow 0}\frac{f(u,r)}{\lambda (B_2(u,r))}\right)^{-1}
		\\
		&= \sup_{z\in X}g(z)\left(\lim_{r\downarrow 0}\frac{\mu(B_2(u,r))}{\lambda (B_2(u,r))}\right)^{-1}
		\\
		&= \sup_{z\in X}g(z)\left(g(u)\right)^{-1}\leq 1.\qedhere
	\end{align*}
\end{proof}

Proposition~\ref{prop:equivalent_definitions_of_weak_modes_finite_dim} indicates the power of the combination of topological density of $E$ and a continuous density with respect to a nice reference measure like Lebesgue measure.

The next lemma shows that the uniformity condition \eqref{eq:uniformity_condition} implies that every $E$-weak mode is an $E$-strong mode.

\begin{lemma}
      \label{lem:uniform_plus_Eweak_implies_supEweak}
      Let $\emptyset \neq E\subset X$, and suppose $u$ is an $E$-weak mode.
      If $(u,E)$ satisfy the uniformity condition \eqref{eq:uniformity_condition}, then $u$ is an $E$-strong mode.
\end{lemma}

\begin{proof}
	From the elementary properties of limits, we have
	\begin{equation*}
		\lim_{r\downarrow 0}\frac{\sup_{z\in u-E}f(z,r)}{f(u,r)}-\lim_{r\downarrow 0}\frac{f(u-v^\ast,r)}{f(u,r)}=\lim_{r\downarrow 0}\frac{\sup_{z\in u-E}f(z,r)-f(u-v^\ast,r)}{f(u,r)}.
	\end{equation*}
	The hypothesis \eqref{eq:uniformity_condition} implies that the right-hand side is zero, since the ratio inside the limit vanishes for all $r<r^{\ast}$.
	Thus the left-hand side is zero, and the limits agree.
	Therefore, given the assumption that $u$ and $E$ satisfy \eqref{eq:Eweak_mode}, it follows that $u$ and $E$ also satisfy \eqref{eq:supE_weak_mode}.
\end{proof}

The proof of Theorem~\ref{thm:equivalence_Eweak_Estrong} follows.
\begin{proof}[Proof of Theorem~\ref{thm:equivalence_Eweak_Estrong}]
	By the right implication in \eqref{eq:strong_implies_supEweak_implies_Eweak}, every $E$-strong mode $u$ is also an $E$-weak mode, regardless of whether or not the pair $(u,E)$ satisfies the uniformity condition \eqref{eq:uniformity_condition}.
	On the other hand, if $(u,E)$ satisfies the uniformity condition, then Lemma~\ref{lem:uniform_plus_Eweak_implies_supEweak} yields that any $E$-weak mode $u$ is also $E$-strong.
\end{proof}

The next example involving a geometric sequence of crossed squares demonstrates that it is possible for a vector $u\in\supp(\mu)$ to satisfy \eqref{eq:Eweak_mode} but not \eqref{eq:supE_weak_mode}, even when $X$ is a finite-dimensional Banach space and $\mu$ is a nonatomic, finite measure on $X$.
The essential idea of the example is to construct a measure such that the ratios of measures of balls decay sufficiently rapidly for \eqref{eq:Eweak_mode} to hold, but not uniformly rapidly, so that \eqref{eq:supE_weak_mode} does not hold.
In addition, although there exist intervals of $r$-values in which the supremum $\sup_{z\in u-E}f(z,r)$ is attained at some element of $u-E$, none of these intervals extend to the origin.
Hence, there does not exist a $v\in E$ such that the values of $\sup_{z\in u-E}f(z,r)$ are dominated by those of $f(u-v,r)$ for all sufficiently small $r$, and the measure $\mu$ constructed below fails to satisfy the uniformity condition.

\begin{figure}
	\centering
	\includegraphics[width=6.5cm]{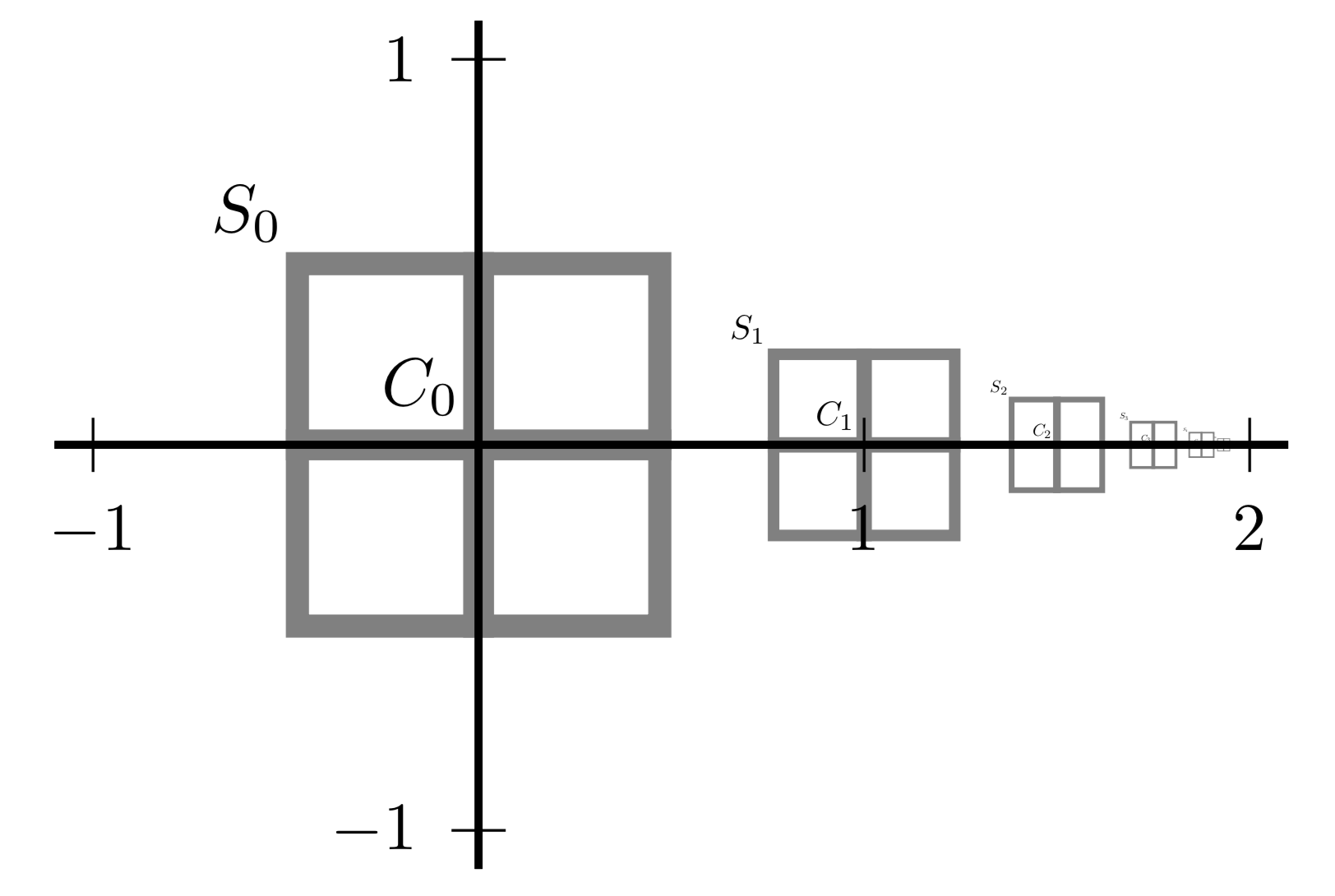}
	\caption{Illustration of the crossed squares $C_{n} \cup S_{n}$ used in the construction of the measure $\mu$ on $\mathbb{R}^{2}$ in Example~\ref{example:weak_does_not_imply_weak_uniform} that has countably infinitely many weak modes but no $\mathbb{Q}^{2}$-strong mode.}
	\label{fig:weak_does_not_imply_weak_uniform}
\end{figure}

\begin{example}[$E$-weak modes that are not $E$-strong]
\label{example:weak_does_not_imply_weak_uniform}
	Let $X=\mathbb{R}^2$, let $E=\mathbb{Q}^2$, let $K(0,1)\defeq B_\infty(0,1)$, and let $u$ be the origin of $X$.
	Fix $\alpha=\tfrac{7}{8}$.
	Define the sequence $(v_n)_{n\in\mathbb{N}_{0}}\subset E$ by
	\begin{equation*}
		v_0=u=(0,0),\quad v_n\defeq \left(\sum^{n-1}_{m=0}2^{-m},0\right).
	\end{equation*}
	For $i = 1, 2$, let $e_i$ denote the $i$\textsuperscript{th} Euclidean basis vector.
	For $n\in\mathbb{N}_{0}$, let 
	\begin{equation*}
		C_{n}^{i}\defeq \left\{v_n+re_i\ \middle\vert\ -\frac{1}{2^{n+1}}<r<\frac{1}{2^{n+1}}\right\}
	\end{equation*}
	be a line segment of length $2^{-n}$ centred at $v_{n}$ that is parallel to the $e_i$-axis.
	Define the cross-shaped set $C_n \defeq C_{n}^{1}\cup C_{n}^{2}$, and let $S_{n} \defeq \partial B_{\infty}\left(v_n,2^{-(n+1)}\right) = \bigcup_{i=1}^{4} S_{n}^{i}$ be the boundary square, where $(S_{n}^{i})_{i=1}^{4}$ are the four edges of the square.
	Thus, as illustrated in Figure~\ref{fig:weak_does_not_imply_weak_uniform}, we have a sequence of crossed squares of the form $C_{n}\cup S_{n}$ centred at $v_{n}$.
	The side lengths of the crossed squares form a geometric sequence, and the sequence of $e_1$-coordinates of the centres give the finite-term approximations of a geometric series converging to $2$.
	Note that, under any $\ell_p$-norm on $X$, the distance from $v_{n}$ to $v_{n - 1}$ is  $2^{-(n-1)}$, and the distance between the corresponding crossed squares is $2^{-(n+1)}$.
	Therefore, $\bigcup_{n\in\mathbb{N}_{0}} (C_{n}\cup S_{n})$ is a disjoint countable union of closed sets that is bounded in $X$.

	Let $\mathcal{H}^1$ denote the one-dimensional Hausdorff measure, normalised so that it assigns to any straight line segment its Euclidean length.
	In particular, $\mathcal{H}^1(S_{n}^{i})=\mathcal{H}^1(C_{n}^{k})$ for any $i=1,\ldots,4$ and $k=1,2$.
	Define the measures $\mu_{n}^{C}$, $\mu_{n}^{S}$, and $\mu_{n}$ by
	\begin{align*}
		\mu_{n}^{C}(B) & \defeq \mathcal{H}^{1}(B \cap C_{n}) = \sum_{i=1}^{2}\mathcal{H}^1 (B\cap C_{n}^{i} ), \\
		\mu_{n}^{S}(B) & \defeq \frac{\alpha}{2} \mathcal{H}^{1}(B \cap S_{n}) = \frac{\alpha}{2} \sum_{i=1}^{4}\mathcal{H}^1 (B\cap S_{n}^{i} ),
		\\
		\mu_{n}(B) & \defeq \mu_{n}^{C}(B) + \mu_{n}^{S}(B)
	\end{align*}
	for any Borel set $B \subseteq \mathbb{R}^2$, so that $\supp(\mu_{n}) = C_{n}\cup S_{n}$.
	Observe that, since $\alpha < 1$,
	\[
		\mu_{n}^{C}(X) = 2^{-(n-1)} > \alpha 2^{-(n-1)} = \mu_{n}^{S}(X) ,
	\]
	and $\mu_{n}(X)=(1+\alpha)2^{-(n-1)}$.
	Let $\mu$ be the measure
	\begin{equation}
		\label{eq:mu_for_which_weak_does_not_imply_weak_uniform}
		\mu(B)\defeq \sum_{n=0}^{\infty}\mu_{n}(B) = \sum_{n = 0}^{\infty} \mathcal{H}^{1} (B \cap C_{n}) + \frac{\alpha}{2} \sum_{n = 0}^{\infty} \mathcal{H}^{1} (B \cap S_{n})
	\end{equation}
	for any Borel set $B \subseteq \mathbb{R}^2$.
	Observe that $\mu$ is a finite measure that is not absolutely continuous with respect to Lebesgue measure.

	We now show that every element in the set $(v_n)_{n\in\mathbb{N}_{0}}$ is an $E$-weak mode, i.e.\ that for every $n\in \mathbb{N}_{0}$ we have
	\begin{equation*}
		\lim_{r\downarrow 0}\frac{\mu(B_\infty(v_{n}-v,r))}{\mu(B_\infty(v_{n},r))}\leq 1,\quad\forall v\in E.
	\end{equation*}
	Note that for every $n\in \mathbb{N}_{0}$, there exists some $R(n)$ such that $\mu(B_\infty(v_{n},r))=4r$ for all $r<R(n)$.

	Let $n\in\mathbb{N}_{0}$ be arbitrary.
	It suffices to consider $v\in E$ such that $v_{n}-v$ belongs to $\bigcup_{n\in\mathbb{N}_{0}}(C_{n}\cup S_{n})$, since otherwise $\mu(B_\infty(v_{n}-v,r))=0$ for all sufficiently small $r$, and the limit above is zero.
	If $v_{n}-v\in Z$, then for sufficiently small $r$ the ratio in the limit is 1.
	This leaves three cases to consider.
	Below, we denote the closure of the set $A$ by $\cl{A}$.

	\textit{Case 1.} Suppose $v_{n}-v\in C_{n}^{i}\setminus (v_{n})_{n\in\mathbb{N}_{0}}$ for some $i\in\{1,2\}$.
	For sufficiently small $r$, $\mu(B_\infty(v_{n}-v,r))=\mu_{n}^{C}(B_\infty(v_{n}-v,r)\cap C_{n}^{i})=2r$, so in this case the limit equals $\tfrac{1}{2}$.

	\textit{Case 2.} Suppose $v_{n}-v\in S_{n}^{j}\setminus \cl{C_{n}}$ for some $j\in\{1,\ldots,4\}$.
	For sufficiently small $r$, $\mu(B_\infty(v_{n}-v,r))=\mu_{n}^{S}(B_\infty(v_{n}-v,r)\cap S_{n}^{j})=2\alpha r$, so in this case the limit equals $\tfrac{\alpha}{2}$.

	\textit{Case 3.} Suppose $v_{n}-v\in S_{n}^{j}\cap\cl{C_{n}}$ for some $j$.
	For sufficiently small $r$, 
	\begin{equation*}
	\mu(B_\infty(v_{n}-v,r))=\mu_{n}^{S}(B_\infty(v_{n}-v,r)\cap S_{n}^{j})+\mu_{n}^{C}(B_\infty(v_{n}-v,r)\cap C_{n})=2\alpha r+r,
	\end{equation*}
	so in this case the limit equals $\tfrac{2\alpha+1}{4}<\tfrac{3}{4}$.
	
	Since in all the preceding cases the limit in question is less than or equal to 1, $v_{n}$ is an $E$-weak mode.
	
	Next, we show that $v_{0}$ is not an $E$-strong mode.
	Fix an arbitrary $n\in\mathbb{N}$.
	Since the $\ell_\infty$-radius of the crossed square $C_{n}\cup S_{n}$ is $2^{-(n+1)}$ and the distance of $C_{n}\cup S_{n}$ to the nearest of the two adjacent crossed squares is $2^{-(n+2)}$, it follows that for arbitrary $\tfrac{1}{2^{n+1}}<r<\tfrac{3}{2^{n+2}}$, we have 
	\begin{equation*}
		\mu(B_\infty(v_n,r))=\mu_{n}(X)=(1+\alpha)2^{-(n-1)}=(1+\alpha)\frac{4}{2^{n+1}},
	\end{equation*}
	and for the same values of $r$ we have
	\begin{equation*}
		\frac{4}{2^{n+1}}<\mu(B_\infty(v_{0},r))=\mu_{0}^{C}(B_\infty(v_{0},r))<\frac{6}{2^{n+1}}.
	\end{equation*}
	Computing the ratio and using that $\alpha=\tfrac{7}{8}$ yields
	\begin{equation*}
		\frac{5}{4}=(1+\alpha)\frac{4}{6}<\frac{ \mu(B_\infty(v_{n},r))}{\mu(B_\infty(v_{0},r))}<1+\alpha.
	\end{equation*}
	Since $n\in\mathbb{N}$ was arbitrary, it follows that there exists a countable sequence of intervals of values of $r$ of the form $\tfrac{1}{2^{n+1}}<r<\tfrac{3}{2^{n+2}}$, such that over each interval the ratio $\sup_{z\in v_{0}-E}\mu(B_\infty(z,r))/\mu(B_\infty(v_{0},r))$ cannot be strictly less than $\tfrac{5}{4}$.
	Thus 
	\begin{equation*}
		\lim_{r\downarrow 0}\frac{\sup_{z\in v_{0}-E} f(z,r)}{f(v_{0},r)}\geq \frac{5}{4}>1,
	\end{equation*}
	and so $v_{0}$ is an $E$-weak mode, but not an $E$-strong mode.

	By replacing $v_{0}$ in the preceding analysis with $v_{m}$ for an arbitrary $m\in\mathbb{N}$ and making appropriate modifications, it follows that the set of $E$-weak modes coincides exactly with the set $\{ v_{n} \}_{n\in\mathbb{N}_{0}}$, and that no element in this set is an $E$-strong mode.
	Thus, $\mu$ has countably infinitely many weak modes, but no $E$-strong mode.
\end{example}

\begin{remark}
      \label{rem:weak_does_not_imply_weak_uniform}
	Example~\ref{example:weak_does_not_imply_weak_uniform} yields two important observations.
	First, a measure may admit a weak mode but not a mode; this follows from \eqref{eq:strong_implies_supEweak_implies_Eweak}.
	Second, the example demonstrates that there exists a finite measure $\mu$ on a vector space $X$ and some dense subset $E$ of $X$ such that some $u\in\supp(\mu)$ is an $E$-weak mode but not an $E$-strong mode. In addition, the measure $\mu$ does not satisfy the uniformity condition.
	Therefore, the uniformity condition \eqref{eq:uniformity_condition} cannot be removed in general.
	This is not surprising, given that \eqref{eq:supE_weak_mode} describes uniform decay of $f(u-v,r)$ values over all $v\in E$, whereas \eqref{eq:Eweak_mode} allows for non-uniform decay.
\end{remark}

In this section, both Proposition~\ref{prop:equivalent_definitions_of_weak_modes_finite_dim} and Example~\ref{example:weak_does_not_imply_weak_uniform} make heavy use of fundamental properties of Lebesgue measure.
However, many inverse problems involve infinite-dimensional vector spaces, for which there exists no analogue of Lebesgue measure.
While Lemma~\ref{lem:uniform_plus_Eweak_implies_supEweak} provides a sufficient condition that holds for any $X$, and while Example~\ref{example:weak_does_not_imply_weak_uniform} shows that this sufficient condition is sharp in full generality, it may be useful to determine alternative sufficient conditions for an $E$-weak mode to be $E$-strong, when the measure $\mu$ and $E$ have sufficient regularity properties, e.g.\ those assumed by Helin and Burger \cite{helin_burger_map_estimates_2015_published}.

Given a $\sigma$-finite measure $\mu$ that satisfies Assumption~\ref{asmp:context}, and  given an arbitrary vector $v\in X$, define the translated measure $\mu_{v}$ by
\begin{equation}
	\label{eq:translated_measure}
	\mu_{v}(A)=\mu(A-v),
\end{equation}
for each Borel set $A \subseteq X$.
If $\mu_{v}$ is absolutely continuous with respect to $\mu$, i.e.\ if $\mu_{v} \ll \mu$, then denote the corresponding equivalence class of Radon--Nikodym derivatives by $[\tfrac{\mathrm{d}\mu_{v}}{\mathrm{d}\mu}]\in L^{1}(\mu)$; $\tfrac{\mathrm{d}\mu_{v}}{\mathrm{d}\mu}$ shall denote a representative of this equivalence class.
Recalling that $\mathcal{N}(x)$ denotes the set of open neighbourhoods of $x \in X$, let $C(U)$ denote the set of functions that are continuous on $U\in\mathcal{N}(x)$.
For any $x\in X$ and $U\in\mathcal{N}(x)$, define the set
\begin{align}
	\label{eq:set_of_good_translation_vectors}
	T(x)\defeq  \left\{v\in X \,\middle|\, \mu_{v} \ll \mu,\ \exists\left(U, \frac{\mathrm{d}\mu_{v}}{\mathrm{d}\mu}\right)\in\mathcal{N}(x)\times \left[\frac{\mathrm{d}\mu_{v}}{\mathrm{d}\mu}\right]\text{ with }\frac{\mathrm{d}\mu_{v}}{\mathrm{d}\mu}\in C(U)\right\}. 
\end{align}
We have the following result:

\begin{lemma}
	\label{lem:quasi_invariant_vectors_implies_limiting_ratio_equals_density}
	Let $\mu$ be a $\sigma$-finite measure on $X$ and $u\in \supp(\mu)$.
	Suppose that $T(u)$ is nonempty, and let $v\in T(u)$.
	Then 
	\begin{equation}
		\label{eq:limiting_ratio_equals_density}
		\lim_{r\downarrow 0}\frac{f(u-v,r)}{f(u,r)} = \frac{\mathrm{d}\mu_{v}}{\mathrm{d}\mu}(u).
	\end{equation}
\end{lemma}
 
The result and the idea of its proof is similar to that of \cite[Lemma 2]{helin_burger_map_estimates_2015_published}.
In our case, however, $X$ is not a Banach space, and $\mu$ need not be finite.
Two other key differences are that we do not assume the existence of a representative of $[\tfrac{\mathrm{d}\mu_{v}}{\mathrm{d}\mu}]$ that is continuous on all of $X$, and that we do not assume that the continuous representative attains either its infimum or its supremum on $U$.

\begin{proof}[Proof of Lemma~\ref{lem:quasi_invariant_vectors_implies_limiting_ratio_equals_density}]
	For arbitrary $(x,y,r)\in X\times X\times \mathbb{R}_{>0}$, the notation \eqref{eq:evaluation_map_f} for the evaluation map $f$ and the notation \eqref{eq:translated_measure} for the translated measure $\mu_{x}$ yield
	\begin{equation*}
		f(x-y,r)=\mu(K(x-y,r))=\mu_{y}(K(x,r)).
	\end{equation*}
	Next, notice that if $\mu_{y}\ll \mu$, then for any element $\tfrac{\mathrm{d}\mu_{y}}{\mathrm{d}\mu}$ of the equivalence class $[\tfrac{\mathrm{d}\mu_{y}}{\mathrm{d}\mu}]\in L^{1}(\mu)$, we have
	\begin{equation*}
		\mu(K(x-y,r))=\int_{K(x,r)}\frac{\mathrm{d}\mu_{y}}{\mathrm{d}\mu}(z) \, \mu(\mathrm{d}z)\leq \left(\sup_{z\in K(x,r)}\frac{\mathrm{d}\mu_{y}}{\mathrm{d}\mu}(z)\right)\mu(K(x,r)).
	\end{equation*}
	Dividing both sides by $\mu(K(x,r))=f(x,r)$ and using that $\mu(K(x-y,r))=f(x-y,r)$, we obtain
	\begin{equation}
		\label{eq:ratio_fixed_radius_bounded_by_sup_density}
		\frac{f(x-y,r)}{f(x,r)}\leq \sup_{z\in K(x,r)}\frac{\mathrm{d}\mu_{y}}{\mathrm{d}\mu}\mu(z).
	\end{equation}
	Modifying the argument for a lower bound yields
	\begin{equation}
		\label{eq:sandwich}
		\inf_{z\in K(x,r)} \frac{\mathrm{d}\mu_{y}}{\mathrm{d}\mu}(z)\leq \frac{f(x-y,r)}{f(x,r)}\leq \sup_{z\in K(x,r)}\frac{\mathrm{d}\mu_{y}}{\mathrm{d}\mu}(z).
	\end{equation}
	Note that it is not necessary for either the infimum or the supremum to be attained.
	In addition, the infimum (supremum) may be $-\infty$ ($+\infty$).

	Suppose that $u$ and $v$ satisfy the hypotheses.
	In particular, since $u\in\supp(\mu)$ it follows that $f(u,r) > 0$ for all $r>0$.
	Since $v\in T(u)$, there exists a neighbourhood $U\in\mathcal{N}(u)$ and a representative $\tfrac{\mathrm{d}\mu_{v}}{\mathrm{d}\mu}$ of $[\tfrac{\mathrm{d}\mu_{v}}{\mathrm{d}\mu}]\in L^{1}(\mu)$ such that $\tfrac{\mathrm{d}\mu_{v}}{\mathrm{d}\mu}$ is continuous on $U$.
	For the remainder of this proof, we may assume without loss of generality that the radius parameter $r$ of $K(u,r)$ is small enough that $K(u,r)\subset U$.

	To establish finiteness of the supremum in \eqref{eq:sandwich} for $x=u$ and sufficiently small values of $r$, we use the hypothesis that $u$ is an $E$-weak mode.
	Suppose that the supremum is not finite for any sufficiently small value of $r$.
	Then for arbitrary $C>1$, there exists a sequence $(r_{n})_{n\in\mathbb{N}}$ in the interval $(0,r^{\ast})$ that satisfies $r_{n}\downarrow 0$ and $\sup_{z\in K(u,r_{n})}\tfrac{\mathrm{d}\mu_{v}}{\mathrm{d}\mu}(z)>C$ for all $n\in\mathbb{N}$.
	By continuity of $\tfrac{\mathrm{d}\mu_{v}}{\mathrm{d}\mu}$, there exists some $M\in\mathbb{N}$ such that $\tfrac{\mathrm{d}\mu_{v}}{\mathrm{d}\mu}>\tfrac{C}{2}$ on $K(u,r_{m})$ for all $m\geq M$, and thus
	\begin{equation*}
		\mu(K(u-v,r_{m}))=\int_{K(u,r_{m})}\frac{\mathrm{d}\mu_{v}}{\mathrm{d}\mu}(z) \, \mu(\mathrm{d}z)\geq  \frac{C}{2}\mu(K(u,r_{m})),\quad\forall m\geq M.
	\end{equation*}
	Taking the limit as $m\to\infty$ and letting $C>4$ be arbitrary, we obtain a contradiction with \eqref{eq:Eweak_mode}.
	Thus the supremum is finite for sufficiently small $r$.
  
	For the infimum in \eqref{eq:sandwich}, we consider the two possible cases.
	In the first case, there exists a neighbourhood $U'\in \mathcal{N}(u)$ such that $[\tfrac{\mathrm{d}\mu_{v}}{\mathrm{d}\mu} \vert_{U'\cap U}]=[0]\in L^{1}(\mu)$.
	Since $\tfrac{\mathrm{d}\mu_{v}}{\mathrm{d}\mu}$ is continuous on $U'\cap U$, it follows that $\tfrac{\mathrm{d}\mu_{v}}{\mathrm{d}\mu}$ vanishes uniformly on $U'\cap U$.
	Thus for all $r$ sufficiently small such that $K(u,r)\subset U'\cap U$, both the infimum and supremum in \eqref{eq:sandwich} are zero.
	In the second case, for all $U'\in\mathcal{N}(u)$ it holds that $[\tfrac{\mathrm{d}\mu_{v}}{\mathrm{d}\mu}\vert_{U'}]\neq [0]$.
	Therefore, by taking a decreasing sequence of subsets $(U'_{m})_{m\in\mathbb{N}}\subset\mathcal{N}(u)$ such that $U'_{m}\subset U$ and $U'_{m}\downarrow \{u\}$, it follows from the continuity of $\tfrac{\mathrm{d}\mu_{v}}{\mathrm{d}\mu}$ on $U'_{m}\cap U$ and $\mu$-a.e.\ non-negativity of $\frac{\mathrm{d}\mu_{v}}{\mathrm{d}\mu}$ that $\tfrac{\mathrm{d}\mu_{v}}{\mathrm{d}\mu}$ is strictly positive on $U'_{m}\cap U$.
	Hence $\tfrac{\mathrm{d}\mu_{v}}{\mathrm{d}\mu}(u) > 0$.
	Using the latter observation and continuity, we may further assume that $\tfrac{\mathrm{d}\mu_{v}}{\mathrm{d}\mu}$ is bounded away from zero on $K(u,r)$ for all sufficiently small $r$.
  
	Let $(r_{n})_{n\in\mathbb{N}}$ be an arbitrary decreasing sequence contained in the interval $(0,r^{\ast})$, such that $r_{n}\to 0$.
	By the boundedness of $K$, by passing to a subsequence\footnote{Passage to a subsequence is not necessary if $K$ is star-shaped with respect to the origin.} if necessary, we may assume that $(K(u,r_{n}))_{n\in\mathbb{N}}$ is a strictly decreasing sequence of sets, i.e.\ $K(u,r_{n_2})\subset K(u,r_{n_1})$ whenever $n_{1}< n_{2}$.
	Since $r_{n}\to 0$ and $u\in K(u,r_{n})$ for all $n\in\mathbb{N}$, it follows that $K(u,r_{n})\downarrow \{u\}$.
	Combining these observations with the definitions of the supremum and infimum, it follows that $(\sup_{z\in K(u,r_{n})}\tfrac{\mathrm{d}\mu_{v}}{\mathrm{d}\mu}(z))_{n\in\mathbb{N}}$ and $(\inf_{z\in K(u,r_{n})}\tfrac{\mathrm{d}\mu_{v}}{\mathrm{d}\mu}(z))_{n\in\mathbb{N}}$ are strictly decreasing and increasing sequences respectively.
	In addition, the first (resp.\ second) sequence is bounded from below (resp.\ above) by $\tfrac{\mathrm{d}\mu_{v}}{\mathrm{d}\mu}(u)$.
	By continuity of $\tfrac{\mathrm{d}\mu_{v}}{\mathrm{d}\mu}$ on $K(u,r_{n})$ for every $n\in\mathbb{N}$, it follows that both sequences converge to $\tfrac{\mathrm{d}\mu_{v}}{\mathrm{d}\mu}(u)$.
	By \eqref{eq:sandwich}, the desired conclusion follows.
\end{proof}

\begin{corollary}
	\label{cor:quasi_invariant_vectors_implies_limiting_ratio_equals_density}
	Let $\mu$ be a $\sigma$-finite measure on $X$ and $u\in\supp(\mu)$.
	Suppose that $T(u)$ is nonempty, and let $E\subset T(u)$ be nonempty.
	If $u$ is an $E$-weak mode, then $\sup_{v\in E}\tfrac{\mathrm{d}\mu_{v}}{\mathrm{d}\mu}(u)\leq 1$.
\end{corollary}

\begin{proof}
 	By Definition~\ref{def:Eweak_mode} of an $E$-weak mode and by Lemma~\ref{lem:quasi_invariant_vectors_implies_limiting_ratio_equals_density}, $\tfrac{\mathrm{d}\mu_{v}}{\mathrm{d}\mu}(u)\leq 1$ for all $v\in E$.
 	By definition of the supremum, the conclusion follows.
\end{proof}

The following result establishes an important sufficient condition for an $E$-weak mode to be $E$-strong:
  
\begin{proposition}
    \label{prop:necessary_condition_Eweak_implies_E_strong}
   Let $\mu$ be a $\sigma$-finite measure on $X$ and $u\in \supp(\mu)$. Suppose that $T(u)$ is nonempty, and let $E\subset T(u)$ be nonempty. Suppose $u$ is an $E$-weak mode. If $\sup_{v\in E}\tfrac{\mathrm{d}\mu_{v}}{\mathrm{d}\mu}(u)$ is continuous at $u$, then $u$ is an $E$-strong mode.
\end{proposition}

\begin{proof}
  For arbitrary $r>0$, $v\in E\subset T(u)$ and a fixed representative $\tfrac{\mathrm{d}\mu_{v}}{\mathrm{d}\mu}$, it follows from \eqref{eq:ratio_fixed_radius_bounded_by_sup_density} and the fact that suprema commute that
  \begin{equation*}
   \frac{\sup_{v\in E}f(u-v,r)}{f(u,r)}\leq \sup_{v\in E}\sup_{z\in K(u,r)}\frac{\mathrm{d}\mu_{v}}{\mathrm{d}\mu}(z)=\sup_{z\in K(u,r)}\sup_{v\in E}\frac{\mathrm{d}\mu_{v}}{\mathrm{d}\mu}(z).
  \end{equation*}
  Recall that $E\subset T(u)$ and that $\tfrac{\mathrm{d}\mu_{v}}{\mathrm{d}\mu}$ is continuous at $u$, for all $v\in T(u)$. Since the supremum of a collection of functions that are lower semicontinuous at a point is itself lower semicontinuous at the point, it follows that $\sup_{v\in E}\tfrac{\mathrm{d}\mu_{v}}{\mathrm{d}\mu}$ is lower semicontinuous at $u$. Thus, we have
  \begin{equation}
	\label{eq:inequality_for_equivalence_of_Eweak_Estrong}
	\sup_{v\in E}\frac{\mathrm{d}\mu_{v}}{\mathrm{d}\mu}(u)\leq \lim_{r\downarrow 0}\sup_{z\in K(u,r)}\sup_{v\in E}\frac{\mathrm{d}\mu_{v}}{\mathrm{d}\mu}(z),
  \end{equation}
  with equality if and only if $\sup_{v\in E}\tfrac{\mathrm{d}\mu_{v}}{\mathrm{d}\mu}$ is continuous at $u$. The conclusion then follows by the hypotheses and Corollary~\ref{cor:quasi_invariant_vectors_implies_limiting_ratio_equals_density}.
\end{proof}
  
In order to show that an $E$-weak mode is $E$-strong, it suffices to show the continuity of $\sup_{v\in E}\tfrac{\mathrm{d}\mu_{v}}{\mathrm{d}\mu}$ at $u$.
Using the fact that every lower semicontinuous function can be obtained as the limit of an increasing sequence of continuous functions, let $(v_{n})_{n\in\mathbb{N}}$ be a sequence in $E$.
Define a sequence of functions $(g_{m})_{m\in\mathbb{N}}$ by
\begin{equation*}
   g_{m}(x)\defeq \max\left\{\frac{\mathrm{d}\mu_{v_{n}}}{\mathrm{d}\mu}(x) \,\middle|\, n\in\mathbb{N},\ n\leq m\right\},\quad\forall x\in X.
\end{equation*}
Note that for any $m\in\mathbb{N}$, there is an open neighbourhood $U_{m}\in\mathcal{N}(u)$ such that $g_{m}$ is continuous on $U_{m}$; simply take $U_{m}\defeq \bigcap_{n\leq m}U_{v_{n}}$, where $\tfrac{\mathrm{d}\mu_{v_{n}}}{\mathrm{d}\mu}$ is continuous on $U_{v_{n}}\in\mathcal{N}(u)$.
Since we will take the limit as $m\to\infty$, we will need that $\bigcap_{n\in\mathbb{N}}U_{v_{n}} \in\mathcal{N}(u)$, i.e.\ there exists some $W\in\mathcal{N}(u)$ such that for all $n\in\mathbb{N}$, $\tfrac{\mathrm{d}\mu_{v_{n}}}{\mathrm{d}\mu}$ is continuous on $W$.
By choosing the sequence $(v_{n})_{n\in\mathbb{N}}\subset E\subset T(u)$ appropriately, we may assume without loss of generality that $(g_{m})_{m\in\mathbb{N}}$ converges to $\sup_{v\in E}\tfrac{\mathrm{d}\mu_{v}}{\mathrm{d}\mu}$ on $W$.
  
In order for $\sup_{v\in E}\tfrac{\mathrm{d}\mu_{v}}{\mathrm{d}\mu}$ to be continuous on $W$, we need that $(g_{m})_{m\in\mathbb{N}}$ converges uniformly on $W$.
The standard result for guaranteeing uniform convergence is the Arzel\`{a}--Ascoli theorem.
In order to apply the latter, we need that the $(g_{m})_{m\in\mathbb{N}}$ form an equicontinuous and pointwise bounded sequence on a compact Hausdorff space.
In other words, there must exist a \emph{compact} neighbourhood $W$ of $u$ such that $\tfrac{\mathrm{d}\mu_{v_{n}}}{\mathrm{d}\mu}$ is continuous on $W$ for all $n\in\mathbb{N}$.
Since $X$ is a Hausdorff topological vector space by Assumption~\ref{asmp:context}, the existence of a point $u$ and a compact neighbourhood $W\in\mathcal{N}(u)$ imply that $X$ is in fact locally compact.
However, it is known that every locally compact Hausdorff topological vector space over $\mathbb{R}$ or $\mathbb{C}$ is finite-dimensional.
Thus, for infinite-dimensional $X$, the approach described above fails in general.
  
Another approach is to avoid using the commutativity of suprema, and instead to consider the special case where one can switch the order in which the $r$-limit and supremum over $v\in E$ are taken. In this case,
\begin{align*}
   \lim_{r\downarrow 0}\frac{\sup_{v\in E}f(u-v,r)}{f(u,r)}&\leq \lim_{r\downarrow 0} \sup_{v\in E}\sup_{z\in K(u,r)}\frac{\mathrm{d}\mu_{v}}{\mathrm{d}\mu}(z)
   \\
   &= \sup_{v\in E}\lim_{r\downarrow 0}\sup_{z\in K(u,r)}\frac{\mathrm{d}\mu_{v}}{\mathrm{d}\mu}(z)=\sup_{v\in E}\frac{\mathrm{d}\mu_{v}}{\mathrm{d}\mu}(u),
\end{align*}
where the last equation follows from continuity of $\tfrac{\mathrm{d}\mu_{v}}{\mathrm{d}\mu}$ for all $v\in E$.
Note that this case allows for the neighbourhoods of continuity of $\tfrac{\mathrm{d}\mu_{v}}{\mathrm{d}\mu}$ to differ, which is advantageous.
A sufficient condition for the commutativity of the $r$-limit and supremum over $v\in E$ is that, for some $U^{\ast}\in\mathcal{N}(u)$ and $v^{\ast}\in E$,
\begin{equation}
	\label{eq:uniformity_condition_density_version}
	\frac{\mathrm{d}\mu_{v^{\ast}}}{\mathrm{d}\mu}(z)= \sup_{v\in E}\frac{\mathrm{d}\mu_{v}}{\mathrm{d}\mu}(z),\quad\forall z\in U^{\ast}.
\end{equation}
The lemma below describes a special case in which \eqref{eq:uniformity_condition_density_version} is satisfied, without requiring continuity of $\tfrac{\mathrm{d}\mu_{v}}{\mathrm{d}\mu}$ for any $v$. It also shows that \eqref{eq:uniformity_condition_density_version} implies the uniformity condition \eqref{eq:uniformity_condition}.
\begin{lemma}
   \label{lem:sup_RN_derivatives_continuous}
   Let $\mu$ be a $\sigma$-finite measure on $X$ and $u\in\supp(\mu)$.
   Let $E$ contain the origin, and suppose that for every $v\in E$, it holds that $\mu_{v}\ll \mu$.
   If there exists some $U^{\ast}\in\mathcal{N}(u)$ such that $\tfrac{\mathrm{d}\mu_{v}}{\mathrm{d}\mu}\leq 1$ on $U^{\ast}$ for all $v\in E$, then $\sup_{v\in E}\tfrac{\mathrm{d}\mu_{v}}{\mathrm{d}\mu}(u)$ is continuous at $u$.
   In addition, \eqref{eq:uniformity_condition} holds.
\end{lemma}
\begin{proof}
	If $v^{\ast}$ is the origin, then $\tfrac{\mathrm{d}\mu_{v^{\ast}}}{\mathrm{d}\mu}$ is constant and equal to 1 on all of $X$.
	By the hypotheses, we have that $\sup_{v\in E}\tfrac{\mathrm{d}\mu_{v}}{\mathrm{d}\mu}$ is constant, hence continuous, on $U^{\ast}$.

	For the second statement, let $r^{\ast}>0$ be sufficiently small so that $K(u,r)\subset U^{\ast}$ for all $r<r^{\ast}$. Since 
	\begin{equation*}
		\frac{\mathrm{d}\mu_{v}}{\mathrm{d}\mu}(z)\leq\sup_{v\in E}\frac{\mathrm{d}\mu_{v}}{\mathrm{d}\mu}(z)=\frac{\mathrm{d}\mu_{v^{\ast}}}{\mathrm{d}\mu}(z)=1,\quad\forall z\in U^{\ast},
	\end{equation*}
	it follows that for all $v \in E$,
	\begin{equation*}
		f(u-v,r)=\int_{K(u,r)}\frac{\mathrm{d}\mu_{v}}{\mathrm{d}\mu}(z) \, \mu(\mathrm{d}z) \leq \mu(K(u,r))=f(u,r),\quad\forall r<r^{\ast}.
	\end{equation*}
	Since equality is attained for $v=v^{\ast}$ being the origin, it follows that \eqref{eq:uniformity_condition} holds.
\end{proof}
   
In this section, we saw that when $X$ is an infinite-dimensional Hausdorff topological vector space, the property that $\mu$ and $E$ are such that $\mu$ has absolutely continuous translates $\mu_{v}$ for all $v\in E$ does not suffice for every $E$-weak mode to be $E$-strong. The significance of Lemma~\ref{lem:sup_RN_derivatives_continuous} is to provide an additional condition on the Radon--Nikodym derivatives such that the desired equivalence holds, and to emphasise the importance of the uniformity condition \eqref{eq:uniformity_condition}. Indeed, the key observations of this section are that the uniformity condition is a sufficient condition, and that in general it cannot be weakened; see Example~\ref{example:weak_does_not_imply_weak_uniform} and Remark~\ref{rem:weak_does_not_imply_weak_uniform}.
    
\section{Discussion}
\label{sec:discussion}

In this section, $\mathbb{K}_{\times} \defeq \mathbb{K} \setminus \{0\}$, where $\mathbb{K} \defeq \mathbb{R}$ or $\mathbb{C}$.

\subsection{Arbitrary non-zero scaling factors}
\label{ssec:arbitrary_nonzero_scaling_factors}

As noted in Section~\ref{sec:introduction}, the fact that we restrict the scaling factor $r$ in \eqref{eq:evaluation_map_f} to $\mathbb{R}_{>0}$ even when the base field $\mathbb{K}$ of $X$ is $\mathbb{C}$ implies that the preceding results apply only to the real restriction of $X$.
However, a careful examination of the proofs indicates that the restriction to strictly positive $r\in\mathbb{R}$ is not necessary, even when $X$ is a real topological vector space.
This is because the set $K$ in \eqref{eq:evaluation_map_f} is a bounded neighbourhood of the origin, by Assumption~\ref{asmp:context}.
In particular, $\rho K$ is a bounded neighbourhood of the origin for any $\rho\in\mathbb{K}_{\times}$, for both $\mathbb{K}=\mathbb{R}$ and $\mathbb{K} = \mathbb{C}$.
We therefore define, analogously to \eqref{eq:evaluation_map_f},
\begin{equation}
	\label{eq:evaluation_map_f_arbitrary_nonzero_scaling}
	J(x, \rho)\defeq x+\rho K,\quad \phi(x, \rho)\defeq\mu(J(x, \rho)) \geq 0,
\end{equation}
for $J\in\mathcal{N}(0)$ and $(x, \rho) \in X \times \mathbb{K}_{\times}$.
Let $\abs{\rho}$ denote the complex modulus of $\rho \in \mathbb{C}$.
Below, we adapt the definitions from Section~\ref{sec:main_definitions_results} to this setting.

\begin{definition}
	\label{def:mode_weak_mode_generalisation}
	A \emph{strong mode} (or simply \emph{mode}) of $\mu$ is any $u \in X$ such that 
	\begin{equation}
		\label{eq:mode_generalisation}
		\lim_{\abs{\rho} \to 0} \frac{\sup_{z \in X} \phi(z, \rho)}{\phi(u, \rho)} = 1.
	\end{equation}
	An \emph{$E$-strong mode} of $\mu$ is a point $u\in\supp(\mu)$ such that 
	\begin{equation}
		\label{eq:supE_weak_mode_generalisation}
		\lim_{\abs{\rho}\to 0}\frac{\sup_{z\in u-E}\phi(z,\rho)}{\phi(u,\rho)}\leq 1.
	\end{equation}
	An \emph{$E$-weak mode} (or simply \emph{weak mode}) of $\mu$ is a point $u \in\text{supp}(\mu)$ such that 
	\begin{equation}
		\label{eq:Eweak_mode_generalisation}
		\lim_{\abs{\rho} \to 0} \frac{\sup_{z \in u - E} \phi(z, \rho)}{\phi(u, \rho)} \leq 1.
	\end{equation}
\end{definition}
From the definitions, we obtain the same chain of implications \eqref{eq:strong_implies_supEweak_implies_Eweak} when we allow $\rho\in\mathbb{K}_{\times}$ as when we restricted $\rho$ to $\mathbb{R}_{>0}$.

The following result, analogously to Corollary~\ref{cor:main_result}, answers Helin and Burger's question for vector spaces over $\mathbb{K}$:

\begin{theorem}
	\label{thm:main_result_general_multiplicative_field_K}
	Under Assumption~\ref{asmp:context}, if $E$ is topologically dense in $X$, $u\in \supp(\mu)$, and the pair $(u,E)$ satisfies
	\begin{align*}
		\exists (v^{\ast},r^{\ast})\in E\times (0,1) &\text{ such that, for all $\rho \in \left\{z\in\mathbb{C}~ \middle\vert~ 0<\abs{z}<r^{\ast}\right\}$,} \\		\notag
		\phi(u-v^\ast,\rho) &=\sup_{z\in u-E}\phi(z,\rho),
	\end{align*}
	then $u$ is an $E$-weak mode if and only if $u$ is a strong mode.
\end{theorem}
The second hypothesis is an extension of the uniformity condition \eqref{eq:uniformity_condition} to $\mathbb{K}_{\times}$.
\begin{proof}
The proof of Proposition~\ref{prop:E_topologically_dense_implies_suprema_of_f_on_r_fibres_coincide} applies for a fixed scaling.
Therefore, after modifying these proofs by replacing $f$, $r$, and small-$r$ limits with $\phi$, $\rho$, and small $\abs{\rho}$-limits respectively, we obtain that $E$-strong mode of $\mu$ in the sense of \eqref{eq:supE_weak_mode_generalisation} also is a strong mode in the sense of \eqref{eq:mode_generalisation}.

To prove the remaining equivalence, suppose that $u$ is an $E$-weak mode, and let $v^{\ast}\in E$ be as in the $\mathbb{K}_{\times}$-uniformity condition. Observe as in the proof of Lemma~\ref{lem:uniform_plus_Eweak_implies_supEweak} that
\begin{equation*}
 \lim_{\abs{\rho}\downarrow 0}\frac{\sup_{z\in u-E}\phi(z,\rho)}{\phi(u,\rho)}-\lim_{\abs{\rho}\downarrow 0}\frac{\phi(u-v^\ast,\rho)}{\phi(u,\rho)}=\lim_{\abs{\rho}\downarrow 0}\frac{\sup_{z\in u-E}\phi(z,\rho)-\phi(u-v^\ast,\rho)}{\phi(u,\rho)},
\end{equation*}
and that the right-hand side is zero by the $\mathbb{K}_{\times}$-uniformity condition. Thus if $u$ is an $E$-weak mode, it is also an $E$-strong mode. Using \eqref{eq:strong_implies_supEweak_implies_Eweak} completes the proof.
\end{proof}

\subsection{Local modes}
\label{ssec:local_modes}

In \cite[Definition 2.5]{agapiou_et_al_sparsity}, one defines a point $u$ in a Banach space $X$ to be a local mode and local $E$-weak mode by allowing translations to points in a norm ball centred at $u$ and the intersection of this ball with $E$.

\begin{definition}
	\label{def:local_weak_strong_modes}
	Let $V\in\mathcal{N}(0)$ be bounded.
	A point $u \in X$ is a \emph{local strong mode} of $\mu$ with respect to $V$ if 
	\begin{equation}
		\label{eq:local_mode}
		\lim_{\abs{\rho} \downarrow 0} \frac{\sup_{z \in u-V}\phi(z, \rho)}{\phi(u, \rho)}=1.
	\end{equation} 
	A \emph{local $E$-weak mode} of $\mu$ with respect to $V$ is a point $u \in\text{supp}(\mu)$ such that 
	\begin{equation}
		\label{eq:local_Eweak_mode}
		\lim_{\abs{\rho} \downarrow 0} \frac{\sup_{z \in u-E\cap V}\phi(z, \rho)}{\phi(u, \rho)} \leq 1.
	\end{equation}
\end{definition}
For brevity, we omit defining the a local $E$-strong mode; note that we have the analogous chain of implications as in \eqref{eq:strong_implies_supEweak_implies_Eweak}. The global nature of a point $u$ that satisfies \eqref{eq:mode_generalisation} and \eqref{eq:Eweak_mode_generalisation} is evident by writing the latter in a similar form to \eqref{eq:local_mode} and \eqref{eq:local_Eweak_mode}, in which case $V=X$.

In \cite[Theorem~2.10]{agapiou_et_al_sparsity}, it was shown that for any log-concave probability measure $\mu$ on separable Banach space $X$, every local strong mode is a global strong mode, and every local $E$-weak mode is also a global $E$-weak mode, for the case when $X$ is a Banach space and $V$ is the unit norm ball.
We now consider the question of when a local $E$-weak mode is a local strong mode.

\begin{theorem}
	Suppose that Assumption~\ref{asmp:context} holds, and let $V\in\mathcal{N}(0)$ be bounded.
	If $E$ is topologically dense in $V$, and if 
	\begin{align*}
		\exists (v^{\ast},r^{\ast})\in E\cap V\times (0,1) &\text{ such that, for all $\rho \in \left\{z\in\mathbb{C}~ \middle\vert~ 0<\abs{z}<r^{\ast}\right\}$,} \\		\notag
		\phi(u-v^\ast,\rho) &=\sup_{z\in u-E\cap V}\phi(z,\rho),
	\end{align*}
	then $u$ is a local $E$-weak mode if and only if $u$ is a local strong mode \eqref{eq:local_mode}.
\end{theorem}
We sketch the proof below.
\begin{proof}
If, in the proof of Proposition~\ref{prop:E_topologically_dense_implies_suprema_of_f_on_r_fibres_coincide}, we replace $X$ with $x-V$, $x - E$ with $x - E\cap V$, and $f(\quark,r)$ with $\phi(\quark,\rho)$, then the density of $E$ in $V$ implies that 
\begin{equation*}
 \sup_{z\in x-V}\phi(z,\rho)=\sup_{z\in x-E\cap V}\phi(z,\rho),
\end{equation*}
for all $(x, \rho) \in X \times \mathbb{K}_{\times}$.
Thus, the local coincident limiting ratios condition with respect to $V$ holds, i.e.
\begin{equation*}
	\lim_{\abs{\rho} \downarrow 0} \frac{\sup_{z \in x-E\cap V}\phi(z, \rho)}{\phi(x, \rho)}=\lim_{\abs{\rho} \downarrow 0} \frac{\sup_{z \in x-V}\phi(z, \rho)}{\phi(x, \rho)},
\end{equation*}
so local $E$-strong modes and local strong modes are equivalent. Using the local uniformity condition hypothesis and elementary properties of limits, we have
\begin{equation*}
	\lim_{\abs{\rho} \downarrow 0} \frac{\sup_{z \in x-V}\phi(z, \rho)}{\phi(x, \rho)}-\lim_{\abs{\rho}\downarrow 0}\frac{\phi(u-v^{\ast},\rho)}{\phi(u,\rho)}=0,
\end{equation*}
which completes the proof that every local $E$-weak mode is a local strong mode. Since the analogue of \eqref{eq:strong_implies_supEweak_implies_Eweak} holds for local modes, the proof is complete.
\end{proof}

\subsection{Examples concerning the existence and $K$-dependence of modes}
\label{ssec:examples_concerning_existence_K-dependence_modes}

We conclude with some further discussion of the existence and $K$-dependency of modes.
Recall that Example~\ref{example:weak_does_not_imply_weak_uniform} gave an example of a measure on $\mathbb{R}^{2}$ with weak modes but no strong modes;
our next example is of a measure on $\mathbb{R}$ with no modes at all:

\begin{figure}[t]
	\begin{center}
		\includegraphics[width=6.0cm]{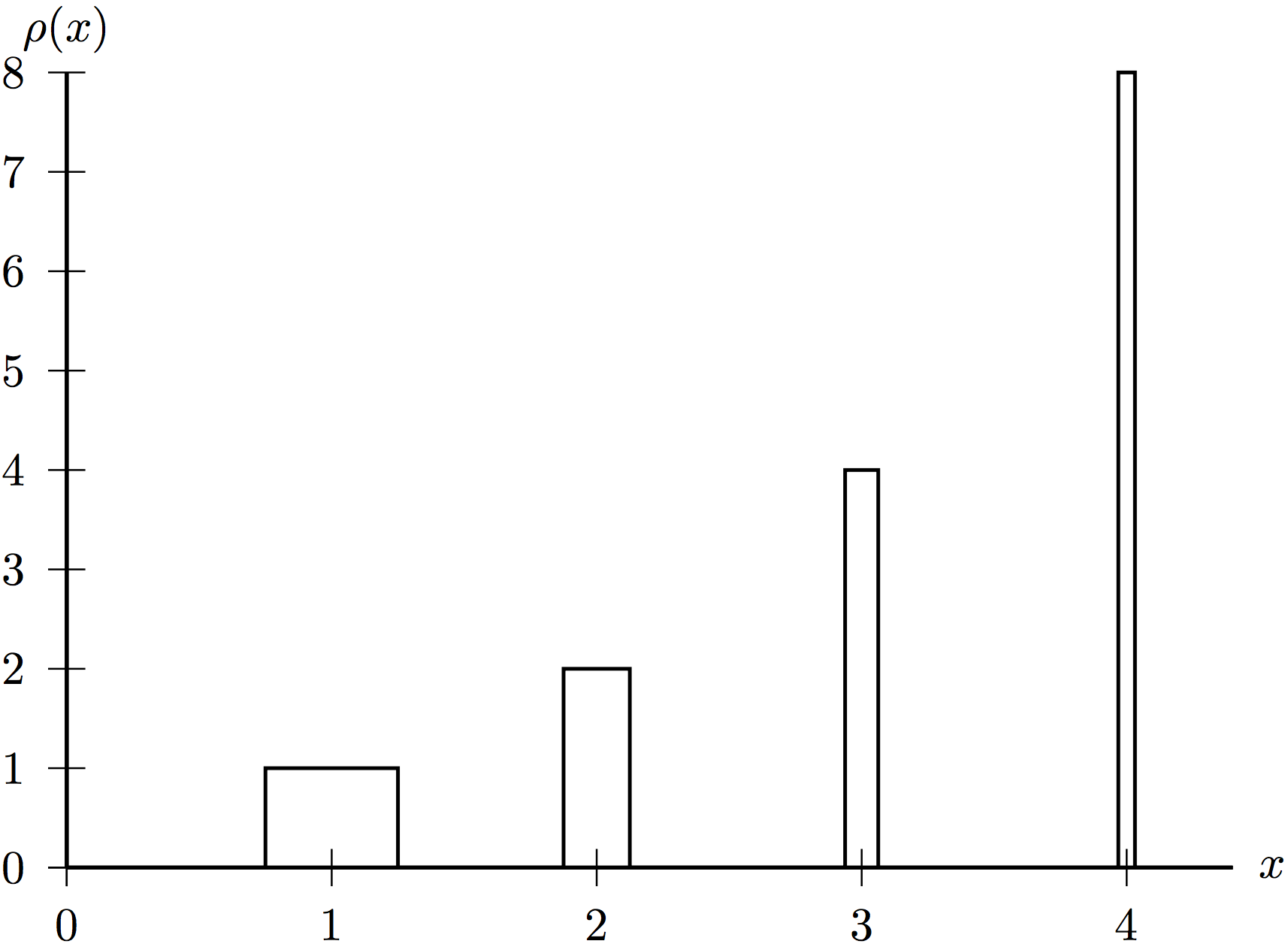}
	\end{center}
	\caption{The probability density function $\rho$ of a probability measure on the real line with no modes in the na{\"\i}ve, strong, or weak senses, given by \eqref{eq:no_mode} with parameters $a = 2$, $b = 4$.
	The local uniform distribution centred on each $n \in \mathbb{N}$ has mass $2^{-n}$, with $\rho$ having magnitude $2^{n - 1}$ on an interval of length $2^{1 - 2 n}$.}
	\label{fig:no_mode}
\end{figure}

\begin{example}[Probability measure without a mode]
	\label{eg:no_mode}
	Fix parameters $1 < a < b < \infty$ with $b > 2$;
	write $A \defeq \sum_{n \in \mathbb{N}} a^{-n} < \infty$.
	A good choice, which helps with the intuition, is to take $a = 2$ and $b = 3$ or $4$.
	Define $\mu$ to be the probability measure with Lebesgue density (probability density function) $\rho \colon \mathbb{R} \to \mathbb{R}_{\geq 0}$,
	\begin{equation}
		\label{eq:no_mode}
		\rho(x) \equiv \frac{\rd \mu}{\rd x} (x) \defeq \sum_{n \in \mathbb{N}} \frac{1}{2 A} (b / a)^{n} \mathbb{I}\bigl[ | x - n | < b^{-n} \bigr] .
	\end{equation}
	As illustrated in Figure~\ref{fig:no_mode}, $\rho$ is a step function with height $\frac{1}{2 A} (b / a)^{n}$ on each interval of length $2 b^{-n}$ centred on each natural number $n$, so that this interval has mass $a^{-n} / A$;
	the density function $\rho$ takes the value $0$ outside these intervals.
	(The restriction that $b > 2$ ensures that the supports of the indicator functions $\mathbb{I}\bigl[ | x - n | < b^{-n} \bigr]$ in \eqref{eq:no_mode} are pairwise disjoint.)

	As can be seen from inspection, $\rho$ is unbounded above and has no global maximum.
	Thus, $\mu$ has no mode in the na{\"\i}ve sense of a maximiser of its density function.
	Furthermore, the small-balls approach does not rectify the situation.
	For the ball radius $r_{n} = b^{-n} > 0$, the unique maximiser of $x \mapsto \mu((x - r_{n}, x + r_{n}))$ is $x_{n}^{\ast} = n$ (with maximum value $a^{-n} / A$).
	Clearly, the sequence $(x_{n}^{\ast})_{n \in \mathbb{N}}$ does not converge in $\mathbb{R}$, and indeed has no convergent subsequence, so we cannot hope that a mode will arise as an accumulation point of the sequence $(x_{n}^{\ast})_{n \in \mathbb{N}}$ of approximate modes, in the style of \cite{dashti_law_stuart_voss_map_estimators_2013_published}, since this sequence has no such accumulation points.
	
	Indeed, it is possible to directly verify that $\mu$ has no strong modes in the sense of \eqref{eq:mode}.
	For any proposed strong mode $u \in \mathbb{R}$, let $N \defeq \lceil u \rceil + 1$.
	Then, for $r \defeq r_{N} = b^{-N}$ and $x_{N}^{\ast} = N$,
	\[
		\frac{\sup_{z \in \mathbb{R}} f(z, r)}{f(u, r)} \geq \frac{f(x_{N}^{\ast}, r)}{f(u, r)} \geq \frac{A^{-1} a^{-N}}{2 b^{- N} \cdot (2 A)^{-1} (b / a)^{N - 1}} = \frac{a^{-N}}{b^{-1} a^{1 - N}} = \frac{b}{a} > 1 .
	\]
	Hence, $u$ cannot be a strong mode for $\mu$.
	Similarly, $u$ is not an $E$-weak mode for any dense $E \subseteq \mathbb{R}$.
	
	This example can be modified to produce a smooth density $\rho$ with the same pathology.
	Also, under the map $x \mapsto \pi^{-1} \arctan (x)$, this pathological measure is pushed forward to one with compact support in $[0, 1]$;
	in this case, the sequence $(x_{n}^{\ast})_{n \in \mathbb{N}}$ of approximate modes would converge to $1$, which is again not a mode.
\end{example}

We note that Example~\ref{eg:no_mode} relies essentially on the \emph{unboundedness} of the density;
in \cite[Example~2.2]{clason_generalized_modes}, another example is given of a probability measure on $\mathbb{R}$ with no strong mode, which instead relies on \emph{discontinuity} of the density.

In light of Example~\ref{eg:no_mode}, it would be interesting to determine conditions on $\mu$ that guarantee the existence of a mode, or conditions on $\mu$ and $E$ that guarantee the existence of an $E$-weak mode.

The definitions of a strong mode in \eqref{eq:mode} and an $E$-weak mode in \eqref{eq:supE_weak_mode} require a choice of the reference neighbourhood of the origin $K$.
The question then arises as to whether, for distinct choices of $K,K'\in\mathcal{N}(0)$, a mode defined in terms of $K$ is also a mode defined in terms of $K'$.
The following two-dimensional example indicates that this is not always the case:

\begin{figure}
	\centering
	\includegraphics[width=6.0cm]{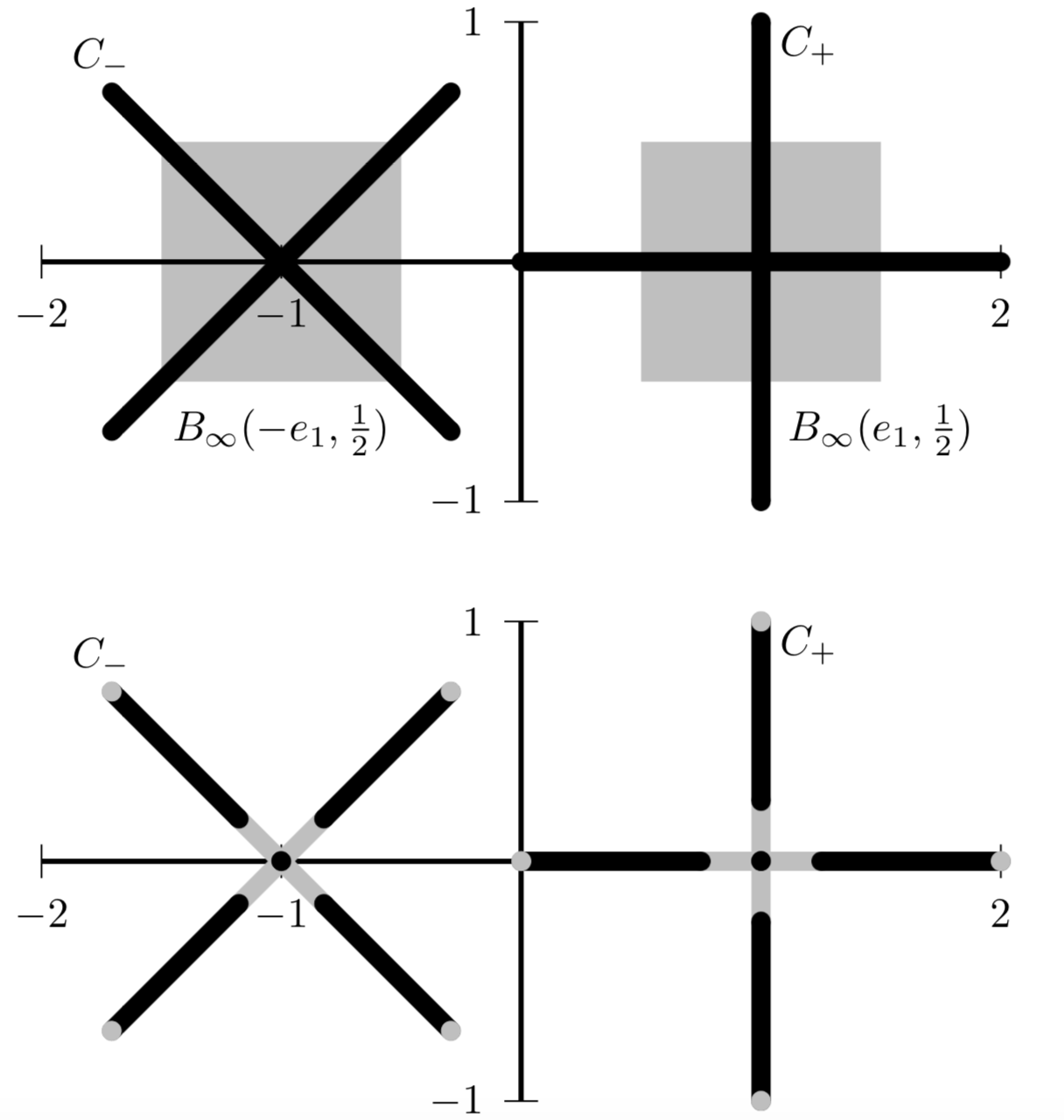}
	\caption{\textsc{Above.}
	Uniform measure $\mu$ on the $1$-dimensional set shown in black has $-e_{1}$ (resp.\ $e_{1}$) as its unique strong mode with respect to $K = B_{\infty}$ (resp.\ $K = B_{1}$), as can be seen by comparing the length of $C_{-} \cap K(-e_{1}, \tfrac{1}{2})$ to that of $C_{+} \cap K(e_{1}, \tfrac{1}{2})$.\newline
	\textsc{Below.}
	The local modes of $\mu$, localised using $V = B_{2}(0, \tfrac{1}{4})$, are the same for both choices of $K$, and are highlighted in black, whereas the remainder of $\supp(\mu)$ is shown in grey.}
	\label{fig:mode_depends_on_K}
\end{figure}

\begin{example}[Modes depend upon $K$]
	\label{eg:mode_depends_on_K}
	Let $K = B_{p}$ be the $\ell_{p}$-norm open unit ball in $X = \mathbb{R}^{2}$;
	we will pay particular attention\footnote{This is a choice of convenience.
	In fact, this example exhibits the same behaviour for any two choices of $0 < p \leq \infty$ on opposite sides of $p = 2$.} to the cases $p = 1$ and $\infty$.
	Let $e_1 \defeq (1,0) \in \mathbb{R}^{2}$,
	\begin{align*}
		C_{-} & \defeq -e_{1} + \{ (x, y) \in \mathbb{R}^{2} \mid x^{2} = y^{2} \leq 1 \} \\
		C_{+} & \defeq e_{1} + \left\{ (x, y) \in \mathbb{R}^{2} \,\middle|\, \begin{array}{c} \text{either } x = 0 \text{ and } | y | \leq 1 \text{,} \\ \text{or } | x | \leq 1 \text{ and } y = 0 \end{array} \right\}, 
	\end{align*}
	and let $\mu(E) \defeq \mathcal{H}^{1}(E \cap (C_{-} \uplus C_{+}))$, where $\mathcal{H}^{1}$ again denotes 1-dimensional Haus\-dorff measure, normalised to give unit measure to any line segment of unit length.
	Thus, $\mu$ is uniform measure on the disjoint union $\supp(\mu) = C_{-} \uplus C_{+}$ of two right-angled crosses, namely $C_{+}$ centred at $e_{1}$ and aligned with the axes and $C_{-}$ centred at $-e_{1}$ and aligned at $\pi / 4$ to the axes;
	$\mu$ has finite total mass $4$.
	
	There are two intuitively plausible candidates for modes of this measure $\mu$, namely $u = \pm e_{1}$.
	Exactly \emph{which} is the strong mode depends whether one chooses $K = B_{1}$ or $K = B_{\infty}$, as illustrated in Figure~\ref{fig:mode_depends_on_K}.
	To be more precise, for $0 < r < \tfrac{1}{2}$,
	\begin{align*}
		\mu(B_{1}(-e_{1}, r)) & = 2 \sqrt{2} r , &
		\mu(B_{\infty}(-e_{1}, r)) & = 4 \sqrt{2} r , \\
		\mu(B_{1}(e_{1}, r)) & = 4 r , &
		\mu(B_{\infty}(e_{1}, r)) & = 4 r ,
		\intertext{and, for any $\tilde{u} \in X \setminus \{ \pm e_{1} \}$ and sufficiently small $r > 0$,}
		\mu(B_{1}(\tilde{u}, r)) & \leq 2 r , &
		\mu(B_{\infty}(\tilde{u}, r)) & \leq 2 \sqrt{2} r .
	\end{align*}
	Thus, $u \mapsto \mu(B_{1}(u, r))$ is globally maximised by taking $u = e_{1}$, whereas $u \mapsto \mu(B_{\infty}(u, r))$ is globally maximised by taking $u = -e_{1}$.
	It then follows that $u = e_{1}$ is a strong mode of $\mu$ with respect to $K = B_{1}$, while $u = -e_{1}$ is a strong mode of $\mu$ with respect to $K = B_{\infty}$.
	These modes are unique since, because $2 \sqrt{2} < 4 < 4 \sqrt{2}$, $u = e_{1}$ is \emph{not} a strong mode of $\mu$ with respect to $B_{\infty}$, and $u = -e_{1}$ is \emph{not} a strong mode of $\mu$ with respect to $B_{1}$.

	The mechanisms underlying this example are twofold:
	first, $\mu$ is not absolutely continuous with respect to $2$-dimensional Lebesgue measure on $X$, and so the Lebesgue differentiation theorem does not apply here;
	secondly, the anisotropy of the $1$-norm and $\infty$-norm matches the anisotropy of the $C_{-}$ and $C_{+}$ components of $\supp(\mu)$.
	This anisotropy is particularly important:
	both choices of $u = \pm e_{1}$ are strong modes of $\mu$ with respect to $K = B_{2}$, the Euclidean ball.
	
	Note also that, for sufficiently small $V$, e.g.\ $V = B_{2}(0, \tfrac{1}{4})$, $u = e_{1}$ becomes a \emph{local} strong mode of $\mu$ with respect to $B_{\infty}$, since this choice of $V$ means that $\mu(B_{\infty}(e_{1}, r))$ is never compared in \eqref{eq:local_mode} against any $\mu(B_{\infty}(z, r))$ with $z \in C_{-}$.
	Similarly, $u = -e_{1}$ becomes a local strong mode of $\mu$ with respect to $B_{1}$.
	Note, however, that the full set of local modes is perhaps counterintuitive:
	it is, as illustrated in Figure~\ref{fig:mode_depends_on_K},
	\[
		\bigl( C_{-} \cap B_{2}(-e_{1}, 1) \cap B_{2}(-e_{1}, \tfrac{1}{4})^{\complement} \bigr) \uplus \bigl( C_{+} \cap B_{2}(e_{1}, 1) \cap B_{2}(e_{1}, \tfrac{1}{4})^{\complement} \bigr) \uplus \{ \pm e_{1} \} ,
	\]
	namely those points of $\supp(\mu)$ that are not at the extremes of the ``arms'' of the crosses, nor too close to the isolated local modes at $\pm e_{1}$.
\end{example}

Example~\ref{eg:mode_depends_on_K} indicates the importance of the choice of the sets $K$ and $J$ in the definitions \eqref{eq:evaluation_map_f} and \eqref{eq:evaluation_map_f_arbitrary_nonzero_scaling} of $f$ and $\phi$ respectively.
In the case when $(X, \Vrt{\quark})$ is a normed vector space, or when $(X, d)$ is a metric vector space, then it may be natural to choose $K$ to be the the unit open $\Vrt{\quark}$-norm or $d$-metric ball, since these sets are native to the structure of the vector space.
If there is no unambiguous choice, then it may be useful to determine whether there exist conditions for which a mode defined in terms of $K$ is also a mode in terms of $K'$.

\section*{Acknowledgements}

We thank Gerd Wachsmuth, in particular for suggesting a simpler proof of a previous version of Theorem~\ref{thm:equivalence_Estrong_strong} and for pointing out an error in a result concerning weak and $E$-strong modes in a previous version of this manuscript.
We also thank Sergios Agapiou and Tapio Helin for their feedback on an early version of this manuscript, and the two anonymous peer reviewers for their helpful comments.

The authors acknowledge support provided by the Freie Universit\"{a}t Berlin within the Excellence Initiative of the German Research Foundation (DFG);
by DFG grant CRC 1114 ``Scaling Cascades in Complex Systems'';
and by the National Science Foundation (NSF) under grant DMS-1127914 to the Statistical and Applied Mathematical Sciences Institute (SAMSI) and SAMSI's QMC Working Group II ``Probabilistic Numerics''.
Any opinions, findings, and conclusions or recommendations expressed in this article are those of the authors and do not necessarily reflect the views of the above-named funding agencies and institutions.

\bibliographystyle{amsplain}
\bibliography{modes_new}

\end{document}